\documentclass[final]{siamltex}
\usepackage{amsmath, amssymb, amsfonts}
\usepackage{graphicx,color,caption,subcaption}
\usepackage{diagbox}
\usepackage{algorithm,algorithmic}
\usepackage{enumitem}
\usepackage{multirow}
\usepackage{verbatim}
\usepackage{hyperref}

\newcommand{\bvec}[1]{\mathbf{#1}}

\newcommand{\mc}[1]{\mathcal{#1}}
\newcommand{\mf}[1]{\mathfrak{#1}}

\newcommand{\vr}{\bvec{r}}

\newcommand{\ud}{\,\mathrm{d}}

\newcommand{\Vion}{V_{\mathrm{ion}}}
\newcommand{\abs}[1]{\lvert#1\rvert}
\newcommand{\norm}[1]{\lVert#1\rVert}

\newcommand{\wt}[1]{\underline{#1}}

\newcommand{\Or}{\mathcal{O}}

\newcommand{\I}{\imath} 

\newcommand{\CC}{\mathbb{C}}

\newtheorem{remark}[theorem]{Remark}
\newtheorem{assumption}{Genericity Assumption}
\newcommand{\proofof}[1]{\textit{Proof of #1}.}

\allowdisplaybreaks

\title{Convergence of adaptive compression methods for
Hartree-Fock-like equations}

\author{
Lin Lin\thanks{Department of Mathematics, University of California, Berkeley, Berkeley, CA 94720 and Computational Research Division, Lawrence Berkeley National Laboratory, Berkeley, CA 94720. Email: \texttt{linlin@math.berkeley.edu}}
\and Michael Lindsey\thanks{Department of Mathematics, University of California, Berkeley, Berkeley, CA 94720. Email: \texttt{lindsey@math.berkeley.edu}}
}

\begin{document}

\maketitle

\begin{abstract}
The adaptively compressed exchange (ACE) method provides an efficient
way for solving Hartree-Fock-like equations in quantum
physics, chemistry, and materials science. The key step of the ACE method
is to adaptively compress an operator that is possibly dense and full-rank.  In this paper, we present a detailed study of 
the adaptive compression operation, and establish rigorous convergence
properties of the adaptive compression method in the context of solving linear
eigenvalue problems.  Our analysis also elucidates the potential use of
the adaptive compression method in a wide range of problems.
\end{abstract}



\pagestyle{myheadings}
\thispagestyle{plain}

\section{Introduction}\label{sec:intro}
 
The Fock exchange operator plays a fundamental role in many-body quantum physics.
The Hartree-Fock
equation (HF)~\cite{SzaboOstlund1989} is the starting point of nearly all wavefunction based
correlation methods in quantum chemistry. Hartree-Fock-like equations also appear in 
the widely used Kohn-Sham density functional theory
(KSDFT)~\cite{HohenbergKohn1964,KohnSham1965} with hybrid
exchange-correlation
functionals~\cite{Becke1993,HeydScuseriaErnzerhof2003,PerdewErnzerhofBurke1996} in quantum chemistry and materials science.
As an example, the B3LYP functional~\cite{Becke1993}, which is only one
specific functional used by KSDFT, has generated more than $60,000$
citations.\footnote{Data from ISI Web of Science, February, 2017.}

Hartree-Fock-like equations require the solution of a large number of eigenpairs of a nonlinear integro-differential operator. From a computational perspective, after linearization and a certain numerical discretization to be detailed later, we solve the following linear eigenvalue problem
\begin{equation}
  (A+B)v_{i} = \lambda_{i} v_{i}, \quad i=1,\ldots,n.
  \label{eqn:lineig}
\end{equation}
Here $A,B\in \CC^{N\times N}$ are Hermitian matrices. The eigenvalues
$\{\lambda_{i}\}$ are real and ordered non-decreasingly. Due to the Pauli
exclusion principle we need to compute the eigenpairs $(\lambda_{i},v_{i})$
corresponding to the lowest $n$ eigenvalues, which are separated from
the rest of the eigenvalues by a positive spectral
gap $\lambda_g:=\lambda_{n+1}-\lambda_{n}$.
Here $n$ encodes the number of electrons in the system, and can range
from tens to tens of thousands. This means that a potentially large number of
eigenpairs need to be computed.
We consider the case that $N$ is large enough so that it is only
viable to use an iterative method to solve \eqref{eqn:lineig}.

In Hartree-Fock-like equations, $A$ in~\eqref{eqn:lineig} is obtained by
discretizing a differential operator involving the Laplace operator.
$B$ is obtained by discretizing the Fock exchange operator, which is an integral operator, and
$B$ is negative definite. 
The discretized Fock exchange operator $B$ is in general a dense full-rank matrix, and it is prohibitively expensive to compute or even to store $B$. 
As one refines the discretization, 
the spectral radius of $A$, denoted by $\norm{A}_{2}$, can become
unbounded, while $\norm{B}_{2}$ remains bounded.  
In an iterative method, one needs to repeatedly apply $(A+B)$ to
some vector $v$, i.e. the matrix-vector multiplication operations
$Av$ and $Bv$ are coupled together. Due to the large spectral radius
$\norm{A}_{2}$, many matrix-vector multiplications may be needed to
reach convergence. 
 In practice, each matrix vector multiplication $Bv$ requires the solution of $n$ Poisson type equations~\cite{GygiBaldereschi1986}, which is far more expensive than computing
$Av$.
Therefore the computational cost of
most iterative solvers will be dominated by the number of matrix-vector
multiplication operations involving $B$.   It is common that the evaluation of $Bv$ alone takes
$95\%$ or more of the overall computational time, which severely limits
the capability of solving Hartree-Fock-like equations for studying
quantum systems of large sizes. In the past decades, there has been a
large amount of work dedicated to reducing the cost of performing
\textit{each} matrix-vector multiplication $Bv$.
This is often done by approximating the dense matrix $B$ by a sparse matrix, which is valid
when the spectral gap $\lambda_{g}$ is
large~\cite{Kohn1996,MarzariVanderbilt1997,Gygi2009,ELiLu2010,OzolinsLaiCaflischEtAl2013,DamleLinYing2015,WuSelloniCar2009,DucheminGygi2010,ChenWuCar2010,DiStasioSantraLiEtAl2014,DawsonGygi2015}.

Recently we have developed an \textit{adaptively compressed exchange operator} formulation (ACE)~\cite{Lin2016ACE}, which reduces the cost for solving Hartree-Fock-like equations from a different and yet more general perspective. 
The key observation is that we only need to find an effective
operator $\wt{B}$ so that $Bv=\wt{B}v$ is satisfied for
$v\in\mathrm{span}~V$, where $V=[v_{1},\ldots,v_{n}]$. $\wt{B}$ can be
constructed to be of strict rank $n$, and hence the computational cost
of $\wt{B}v$ is much smaller than that of $Bv$. 
Note that the subspace
$\mathrm{span}~V$ is precisely the solution for \eqref{eqn:lineig} and
is not known \textit{a priori}. Therefore $\wt{B}$ needs to be constructed in an adaptive
manner. Starting from some initial guess $V^{(0)}$, we will obtain a sequence $V^{(k)}$ and corresponding compressed operators $\wt{B}[V^{(k)}]$. More specifically, our approach is a fixed-point iteration given by
\begin{equation}
  (A+\wt{B}[V^{(k)}])v^{(k+1)}_{i} = \lambda^{(k+1)}_{i} v^{(k+1)}_{i}, \quad i=1,\ldots,n.
  \label{eqn:acelineig}
\end{equation}
Here the operator $\wt{B}$ depends nonlinearly on $V$.
If the sequence of subspaces $\mathrm{span}~V^{(k)}$ converges to $\mathrm{span}~V$, then in the limit the compressed operator $\wt{B}[V^{(k)}]$ will agree with $B$ on $\mathrm{span}~V$, and the eigenvalue problem~\eqref{eqn:lineig} is solved \textit{without loss of accuracy}. 

This paper aims to prove
the convergence properties of this adaptive compression method. At first glance, the advantage of converting a linear eigenvalue
problem~\eqref{eqn:lineig} to a nonlinear eigenvalue
problem~\eqref{eqn:acelineig} is not clear. We will see that 
the adaptive compression method decouples the
matrix-vector multiplication operations $Av$ and $Bv$, and
asymptotically the number of $Bv$ operations is independent of the
spectral radius $\norm{A}_{2}$.

We will
demonstrate that $\wt{B}$ depends only on $\mathrm{span}~V$, so we can consider
the fixed point iteration \eqref{eqn:acelineig} to be a map $P^{(k)}
\mapsto P^{(k+1)}$, where $P^{(k)}$ is the orthogonal projector
$P^{(k)}=V^{(k)}(V^{(k)})^*$.  
Let $\mathbf{H}_N$ denote the set of Hermitian $N\times N$ matrices, and
$\mathcal{D}\subset \mathbb{C}^{N\times N}$ denote the set of rank-$n$
orthogonal projectors on $\mathbb{C}^N$.
The main results of the paper are as follows.

\vspace{2em}

\begin{theorem}[Optimality]
\label{thm:main1}
For $B \prec 0$ and any $N\times n$ matrix $V$ with linearly independent columns, the adaptive compression $\wt{B}[V]$ is the unique
rank-$n$ Hermitian matrix that agrees with $B$ on $\mathrm{span}~V$. Furthermore, $B \preceq \wt{B}[V] \preceq 0$.
\end{theorem}

\vspace{2em}

\begin{remark}
``Optimality'' refers to the minimality of the rank of the adaptive
compression $\wt{B}[V]$ amongst matrices agreeing with $B$ on
$\mathrm{span}~V$. A compression of the lowest possible rank is
desirable in order to minimize the cost of multiplication by $\wt{B}[V]$.
\end{remark}
\vspace{2em}
\begin{remark}
The implication of Theorem \ref{thm:main1} that $\wt{B}[V] \succeq B$ is
significant because it guarantees that the ``bottom-$n$'' eigenspace of $A+B$ is the same as that of $A+\wt{B}[V]$, where $V$ is obtained from the solution of \eqref{eqn:lineig}. (See Lemma \ref{lem:orderzero} below.) This ensures that the fixed-point iteration \eqref{eqn:acelineig} has $V$ as a fixed point, as is necessary for convergence.
\end{remark}

\vspace{2em}

\begin{theorem}[Local convergence]
\label{thm:main2}
For every pair $(A,B)\in \mathbf{H}_N \times \mathbf{H}_N$ with  $B \prec
0$, the fixed point iteration \eqref{eqn:acelineig} converges locally to $P = VV^*$. 
The number of matrix-vector multiplications $Bv$ 
needed for $k$ steps of fixed point iteration is $nk$. Starting from
$P^{(0)}\in \mc{D}$, the asymptotic convergence rate is
\[
\norm{P-P^{(k)}}_2 \lesssim  \gamma^k \norm{P-P^{(0)}}_2, \quad \textnormal{where}\ \ \gamma \leq \frac{\norm{B}_2}{\norm{B}_2+\lambda_g}.
\]
\end{theorem}

\vspace{2em}

\begin{theorem}[Global convergence]
\label{thm:main3}
For almost every pair $(A,B)\in \mathbf{H}_N \times \mathbf{H}_N$ (with
respect to the Lebesgue measure on $\mathbf{H}_N \times \mathbf{H}_N$)
with  $B \prec 0$, the fixed point iteration \eqref{eqn:acelineig}
converges globally to $P = VV^*$ for almost every initial guess
$P^{(0)}\in \mathcal{D}$ (with respect to a natural measure on
$\mathcal{D}$). 
\end{theorem}

\vspace{2em}

\begin{remark}
With minor modification, the condition $B \prec 0$ can be relaxed, so
that the adaptive compression method is applicable to all $B\in
\mathbf{H}_{N}$. See Section \ref{sec:definiteReduction}.
\end{remark}

\vspace{2em}

\begin{remark}
Let $\mathbf{S}_N$ denote the set of real-symmetric $N\times N$
matrices, and $\mathcal{D}_{\mathbb{R}}\subset \mathbb{R}^{N\times N}$
denote the set of rank-$n$ orthogonal projectors on $\mathbb{R}^N$.
Then Theorems~\ref{thm:main2} and~\ref{thm:main3} hold if we replace $\mathbf{H}_N$
with $\mathbf{S}_N$ and $\mathcal{D}$ with $\mathcal{D}_{\mathbb{R}}$.
\end{remark}

\vspace{2em}

In
practice, Eq.~\eqref{eqn:lineig} is only the linearized
Hartree-Fock-like equation, and it is possible to 
employ the flexibility in the adaptive compression formulation by
delaying the update of the compressed operator $\wt{B}$ to further
reduce the number of $Bv$ operations.  This strategy is undertaken in
\cite{Lin2016ACE}. Numerical observation indicates that the ACE
formulation can significantly reduce the number of iterations to solve
Hartree-Fock-like equations, and may reduce the computational time by an
order of magnitude~\cite{Lin2016ACE}. The adaptive compression
formulation has already been adopted by community software packages for
electronic structure calculations such as Quantum
ESPRESSO~\cite{GiannozziBaroniBoniniEtAl2009} for solving
Hartree-Fock-like equations for real materials. 

\subsection{Applicability to nearly degenerate eigenvalue problems}\label{subsec:metals}

Theorem~\ref{thm:main2} suggests that the adaptive compression method
converges fast when the spectral gap $\lambda_{g}$ is large, which is
the case for insulating systems in quantum physics. However,
$\lambda_{g}$ is small for semiconducting systems, and can be virtually
zero for metallic systems. In this case, 
one can compute $n$ eigenvectors, where $n$ is set to be larger than
$m$, the
number of eigenvectors needed in solving Hartree-Fock-like equations.
Although the convergence of the rank-$n$ projector $P^{(k)}$ is expected to be slow,
one is actually only interested in the convergence of the rank-$m$
``sub-projector'' $P_m^{(k)}$ onto the span of the lowest $m$
eigenvectors.  This procedure is rigorously justified in
Theorem~\ref{thm:sub-projector}. We find that the asymptotic convergence
rate of the sub-projector is governed by the gap $\lambda_{n+1}-\lambda_m$, rather than the gap $\lambda_{m+1} -
\lambda_{m}>0$, which is assumed to be positive only to ensure that the
rank-$m$ orthogonal projector $P_m$ is unambiguously defined.

\begin{theorem}[Convergence of sub-projectors]\label{thm:sub-projector}
Let $P^{(k)}$ converge to $P$ (as broadly guaranteed by Theorem \ref{thm:main3}). Then $P_m^{(k)}$ converges to $P_m$ with asymptotic convergence rate given by
\[
\norm{P_m-P_m^{(k)}}_2 \lesssim  \gamma_m^k \norm{P-P^{(0)}}_2, \quad \textnormal{where}\ \ \gamma_m \leq \frac{\norm{B}_2}{\norm{B}_2+\Delta_m}.
\]
Here $\Delta_m = \lambda_{n+1}-\lambda_{m}$.
\end{theorem}

\subsection{Applicability to more general problems}

Although we have Hartree-Fock-like equations in mind throughout the
paper, it is easy to see that the adaptive compression method can be
applied to a  wider variety of 
problems. The case that $B$ is ``small'' and ``costly'' can occur when
$B$ comes from the discretization of an integral operator or a more
general nonlocal operator.  For example, in linear response theories
such as time-dependent density functional theory and Bethe-Salpeter
equations~\cite{RungeGross1984,OnidaReiningRubio2002}, generalized eigenvalue problems arise involving matrices of the form $A+B$, where $A$ is a diagonal matrix and $B$ is a discretized
nonlocal operator with additional structure. Adaptive compression
methods with structure-preserving properties could be applicable to
these problems. The concept of adaptive compression can also be useful
in solving linear equations, as recently demonstrated in the adaptively
compressed polarizability operator formulation for first principle
phonon spectrum calculations~\cite{LinXuYing2017}.
We are currently exploring these directions.

\subsection{Related work}

A Hartree-Fock-like equation, considered as in ~\eqref{eqn:lineig} after linearization
and discretization, constitutes a standard linear eigenvalue problem, and the
present work should be directly compared with existing iterative
eigensolvers, such as the subspace iteration
method~\cite{Parlett1980}, the shift-invert Lanczos
method~\cite{ParlettSaad1987}, the
preconditioned steepest descent method~\cite{Davidson1975},  the preconditioned
conjugate gradient method~\cite{Knyazev2001},  the Jacobi-Davidson method~\cite{SleijpenVanDerVorst2000}, etc. 
In these approaches, the matrix-vector multiplication always
takes the form $(A+B)v$, and the number of $Bv$ operations is $n$ times the
number of iterations. In the absence of a good preconditioner, the
number of iterations in these solvers typically
depends on $\norm{A+B}_{2}$, which is undesirable. Even when a
good preconditioner is available, we still find that the adaptive
compression method can be advantageous, thanks to the flexibility
introduced by decoupling $Av$ and $Bv$ operations. Note that
Eq.~\eqref{eqn:acelineig} is only a fixed point iteration, and the
convergence rate of the adaptive compression method can be further
enhanced by combining with existing acceleration techniques such as
the usage of conjugate directions~\cite{Knyazev2001} and Broyden type
methods~\cite{Anderson1965}. We will report detailed numerical study of
the adaptive compression methods in a forthcoming publication.
We also note that the adaptive
compression method is very simple to implement and only requires a ``black-box'' subroutine for the
computation of $Bv$. Hence in the context of solving Hartree-Fock-like
equations, it is compatible with any existing method that reduces the
cost of the matrix-vector multiplication, such as those using linear
scaling techniques
and using fast solvers for elliptic
equations.

\subsection{Outline of the paper}

The rest of the paper is organized as follows. After presenting a brief introduction
to Hartree-Fock-like equations in Section~\ref{sec:hf}, we introduce the
adaptive compression method in Section~\ref{sec:ace}.
Section~\ref{sec:properties} discusses the properties and optimality of
the compression map $V\mapsto \wt{B}[V]$.  In Section~\ref{sec:local} we
establish the local
convergence with an asymptotic rate, followed
by the global convergence in Section~\ref{sec:global}. Finally, some technical calculations and proofs omitted
in the main text are presented in the appendices.

\section{Hartree-Fock-like equations}\label{sec:hf}

The Hartree-Fock-like equations are a set of nonlinear equations as
follows~\cite{Martin2004}
\begin{equation}
  \begin{split}
    &H[P]\psi_{i} = \left(-\frac12 \Delta  +
    \Vion +
    V_{\text{Hxc}}[P] + V_{X}[P]\right)\psi_{i} = \varepsilon_{i} {\psi}_{i},\\
    &\int {\psi}^{*}_{i}(\vr) {\psi}_{j}(\vr) \ud \vr = \delta_{ij},
    \quad
    P(\vr,\vr') = \sum_{i=1}^{N_e} \psi_i(\vr) \psi_i^*(\vr').
  \end{split}
  \label{eqn:HF}
\end{equation}
Here the eigenvalues $\{\varepsilon_{i}\}$ are ordered non-decreasingly,
and $N_e$ is the number of electrons (spin degeneracy omitted). 
$P$ is the density matrix, which is an orthogonal projector with an
exact rank $N_e$. The diagonal entries of the kernel of $P$ gives the electron density $\rho(\vr)=P(\vr,\vr)$. 
$\Vion$ characterizes the electron-ion interaction in all-electron
calculations. 
$V_{\text{Hxc}}$ is a local operator, and characterizes the Hartree
contribution and the exchange-correlation contribution modeled at a
local or semi-local level. It typically  depends only on the
electron density.  The exchange operator $V_{X}$ is an integral operator
with kernel
\begin{equation}
  V_{X}[P](\vr,\vr') = 
  -P(\vr,\vr') K(\vr,\vr').
  \label{eqn:VXkernel}
\end{equation} 
Here $K(\vr,\vr')$ is the kernel for the electron-electron interaction.
For example, in the Hartree-Fock theory, $K(\vr,\vr')=1/\abs{\vr-\vr'}$
is the Coulomb operator. In screened exchange
theories~\cite{HeydScuseriaErnzerhof2003}, $K$
can be a screened Coulomb operator with kernel
$K(\vr,\vr')=\text{erfc}(\mu\abs{\vr-\vr'})/\abs{\vr-\vr'}$.
$V_X$ is a negative semidefinite operator.  The kernel of $V_X$  is not low
rank due to the Hadamard product (i.e. element-wise product) between the
kernels of $P$ and $K$. From a computational perspective, it is
prohibitively expensive to explicitly construct $V_X[P]$, and it is only
viable to apply it to a vector $v(\vr)$ as
\begin{equation}
  \left(V_{X}[P]v\right)(\vr) = 
  -\sum_{i=1}^{N_{e}} \psi_{i}(\vr) \int
  K(\vr,\vr')\psi_{i}^{*}(\vr')v(\vr') \ud \vr'.
  \label{eqn:applyVX}
\end{equation}
This operation is much more expensive than computing 
$(H[P]-V_{X}[P])v$. In practical Hartree-Fock calculations, the application of $V_X[P]$ to vectors can often take more than $95\%$ of the overall computational time.

The Hartree-Fock-like equations require the density matrix $P$ to be
computed self-consistently.  A common strategy is to solve the
linearized Hartree-Fock equation by fixing the density matrix $P$ so
that $H[P]$ becomes a fixed operator. Then one solves a nonlinear fixed
point problem to obtain the self-consistent $P$. The most time consuming
step is to solve the linearized Hartree-Fock equation.  After numerical
discretization, this gives rise to the linear eigenvalue
problem~\eqref{eqn:lineig}, where $B$ corresponds to the discretized
Fock operator $V_{X}[P]$, and $A$ corresponds to the remaining part
$H[P]-V_X[P]$. We also remark that after numerical discretization,
$B$ is a negative definite matrix. 

\section{Adaptive compression method}\label{sec:ace}

\subsection{Method description}

In order to reduce the number of matrix-vector multiplication operations
$Bv$, the simplest idea is to fix $w_{i}:=B v_{i}$ at some stage, and
to replace $Bv_{i}$ by $w_{i}$ for a number of
iterations. This leads to the following sub-problem
\begin{equation}
  A v_{i} + w_{i} = \lambda_{i} v_{i}, \quad i=1,\ldots,n.
  \label{eqn:fixw}
\end{equation}
Note that Eq.~\eqref{eqn:fixw} is not an eigenvalue problem:
if $v_{i}$ is a solution to \eqref{eqn:fixw},
then $v_{i}$ multiplied by a constant $c$ is typically not a solution.
Eq.~\eqref{eqn:fixw} could be solved using 
optimization based methods, but such a problem is typically more difficult 
than an Hermitian eigenvalue problem. 
In practice, software packages for solving Hartree-Fock-like equations
are typically built around eigensolvers, which is another important 
factor that makes the sub-problem~\eqref{eqn:fixw} undesirable.

The adaptive compression method reuses the information in
$\{w_{i}=Bv_{i}\}$ in a 
different way, which retains the structure of the eigenvalue problem~\eqref{eqn:lineig}. Define $V=[v_{1},\ldots,v_{n}]$,
$W=[w_{1},\ldots,w_{n}]$, so $V,W\in \CC^{N\times n}$, and
construct
\begin{equation}
  \wt{B}[V] = W(W^{*}V)^{-1}W^{*}.
  \label{eqn:Bace}
\end{equation}
Since $B\prec 0$, $W^{*}V\equiv
V^{*}BV$ has only negative eigenvalues and is invertible.  $\wt{B}[V]$ is 
Hermitian of rank $n$, and agrees with $B$ when applied to $V$ as
\begin{equation}
  \wt{B}[V]V = W(W^{*}V)^{-1}W^{*}V = W = BV.
  \label{eqn:Bconsistency}
\end{equation}
We shall refer to the operation from $B$ to $\wt{B}[V]$ as an \textit{adaptive compression}. 

In an iterative scheme, denote by
$V^{(k)}=[v_{1}^{(k)},\ldots,v_{n}^{(k)}]$ the approximate eigenvectors
at the $k$-th iteration of~\eqref{eqn:acelineig}. Then the adaptive
compression method proceeds as follows. 
After $\wt{B}[V^{(k)}]$ is constructed, \eqref{eqn:acelineig} can be solved via
\textit{any} iterative eigensolver to obtain $V^{(k+1)}$. The
iterative eigensolver only requires the application of $A$ and the low
rank matrix $\wt{B}$ to vectors, and does not require any additional
application of $B$ until $V^{(k+1)}$ is obtained. If $\text{span}~V^{(k)}$
converges to  $\text{span}~V$, then the consistency condition
$\wt{B}[V]V = BV$ is satisfied, and the adaptive compression method is numerically exact.
The adaptive compression method for solving the linear eigenvalue
problem~\eqref{eqn:lineig} is given in Algorithm ~\ref{alg:acelineig}, where
we initialize $V^{(0)}$ by solving the eigenvalue problem in the absence
of $B$.

\begin{algorithm}[H]
\caption{The adaptive compression method for solving Eq.~\eqref{eqn:lineig}}

\begin{algorithmic}[1]

  \STATE Initialize $V^{(0)}$ by solving
  $A v_{i}^{(0)} = \lambda_{i}^{(0)} v_{i}^{(0)},\quad i=1,\ldots,n$.

  \WHILE {convergence not reached}

  \STATE Compute $W^{(k)}=BV^{(k)}$.
  \STATE Evaluate $\left[ (W^{(k)})^*
  V^{(k)}\right]^{-1}$ to construct
  $\wt{B}[V^{(k)}]$ implicitly.
  
  \STATE Solve \eqref{eqn:acelineig} to obtain $V^{(k+1)}$.
  
  \STATE Set $k\gets k+1$.
  
  \ENDWHILE

\end{algorithmic}
\label{alg:acelineig}
\end{algorithm}

\subsection{Relaxing the definitiveness condition for $B$}\label{sec:definiteReduction}

As will be seen later, the condition that $B\prec 0$ is important for
the consistency of the adaptive compression method, but this constraint
can be easily relaxed as follows for more general $B$.
Note that replacing $B$ with $B_t := B-t$ (here $t$ as a matrix means the identity
matrix scaled by a real number $t$) in the eigenvalue problem
\eqref{eqn:lineig} yields an equivalent eigenvalue problem, where
all eigenvalues are shifted down by $t$ and the corresponding
eigenspaces are unchanged.  
Thus taking $t > \lambda_{\max} (B)$ ensures that $B_t$ is negative
definite.  
We call this procedure a \textit{$t$-shifted adaptive compression}. 
The spirit of this construction is related to the
``level-shifting'' method used in quantum
chemistry~\cite{SaundersHillier1973}.
Theorem~\ref{thm:main2} suggests that the convergence rate of Algorithm
\ref{alg:acelineig} can be optimized by minimizing $\norm{B_{t}}_{2}$. 
This also opens up the interesting possibility of accelerating the
convergence of the adaptive compression method by taking $t$ to be
negative when $B$ is already negative definite. 
In the discussion below, we will assume that $B$ is negative
definite unless otherwise specified.

\section{Optimality of the adaptive compression}\label{sec:properties}

In \eqref{eqn:Bace}, we have specified how to compress an Hermitian
negative definite matrix $B$ into a rank $n$ matrix with the same
behavior on $\mathrm{span}~V$. 
Since $V$ has orthonormal columns, the orthogonal projector onto $\mathrm{span}~V$ is 
\begin{equation}
  P = V V^{*}.
  \label{eqn:DM}
\end{equation}
In the context of Hartree-Fock-like equations, $P$ in
Eq.~\eqref{eqn:DM} is the discretized density matrix. In the discussion
below, we use the terminology \textit{density matrix} in a slightly more
general sense:

\begin{definition}\label{def:dm}
  For $H\in \mathbf{H}_{N}$ with eigenvalues
  $\{\lambda_{i}\}_{i=1}^{N}$ ordered non-decreasingly and a given
  number $1\le n\le N$, if the spectral gap $\lambda_g :=
  \lambda_{n+1}-\lambda_{n}$ is positive, the density matrix associated
  with $H$ and $n$ is defined to be the orthogonal projector onto the
  span of the first $n$ eigenvectors of $H$. 
\end{definition}

\begin{remark}
In this paper, all density matrices are idempotent, i.e. $P^2=P$. 
When the context is clear, we may drop the dependence on $H$ and $n$ and simply refer to an orthogonal projector $P$ as a
density matrix.
We also let $\mathcal{D}=\mathcal{D}_{\mathbb{C}} \subset
\mathbb{C}^{N\times N}$ denote the set of rank-$n$ density matrices.
\end{remark}

Using the density matrix, the compressed matrix $\wt{B}[V]$ can be
expressed as
\begin{equation}
  \wt{B}[V] = BV(V^{*}BV)^{-1}V^{*}B = B(PBP)^{\dagger}B,
  \label{eqn:Bproj}
\end{equation}
where $(PBP)^{\dagger}$ is the Moore-Penrose
pseudoinverse~\cite{GolubVan2013} of the rank-$n$ matrix $PBP$. 

We elucidate the second equality of \eqref{eqn:Bproj} by examining
its block structure in the matrix representation, as this perspective will be convenient in future developments. Denote by $\{v_i\}_{i=1}^N$ a completion of $\{v_i\}_{i=1}^n$ to an
orthonormal basis of $\mathbb{C}^N$. For any $1\le m\le N$, define
\begin{equation}\label{eqn:subspaceBasisMatrix}
  V_{m} = [v_{1},\ldots, v_{m}].
\end{equation}
In particular, $V_{N}=[v_{1},\ldots,v_{N}]$ consists of all eigenvectors,
and $V\equiv V_{n}=[v_{1},\ldots, v_{n}]$ consists of the eigenvectors
to be computed.
The matrix representation
of $PBP$ with
respect to the basis $V_{N}$ is given in the block form by
\begin{equation}
[PBP]_{V_N} =  \left(\begin{array}{cc}
V^{*}BV & 0\\
0 & 0
\end{array}\right),
\end{equation}
where the size of the upper-left block is $n\times n$. 
Thus the matrix representation of the pseudoinverse $(PBP)^{\dagger}$ is 
\begin{equation}
[(PBP)^{\dagger}]_{V_N} = \left(\begin{array}{cc}
(V^{*}BV)^{-1} & 0\\
0 & 0
\end{array}\right).
\label{eqn:pseudomatrix}
\end{equation}
Hence $(PBP)^\dagger = V(V^* BV)^{-1}V^{*}$, which implies the second equality of \eqref{eqn:Bproj}.
Eq.~\eqref{eqn:Bproj} suggests that 
$\wt{B}$ is a matrix function of the density matrix $P$, or
equivalently, a function of the subspace $\mathrm{Im}(P) =
\mathrm{span}~V$. With some abuse of notation, we will not distinguish between
$\wt{B}[V]$ and $\wt{B}[P]$, and we will mostly use
the projector formulation $\wt{B}[P]$ in the discussion below.

Denote by
\begin{equation}
  [B]_{V_{N}} = \left(\begin{array}{cc}
B_{11} & B_{12}\\
B_{12}^{*} & B_{22}
\end{array}\right)
\label{eqn:Bblock}
\end{equation}
the matrix representation of $B$, and
$B_{11} = V^{*}BV$. Then \eqref{eqn:Bproj} and \eqref{eqn:pseudomatrix} 
give the matrix representation of $\wt{B}[P]$ as
\begin{equation}
  \left[\wt{B}[P]\right]_{V_{N}} = \left(\begin{array}{cc}
B_{11} & B_{12}\\
B_{12}^{*} & B_{12}^{*}B_{11}^{-1}B_{12}
\end{array}\right).
\label{eqn:tildeBblock}
\end{equation}
Note that only the lower-right matrix block is changed in the adaptive
compression.

\begin{remark}[Smoothness of adaptive compression]\label{rem:compressionContinuity}
We can rewrite \eqref{eqn:Bproj} as 
\[
\wt{B}[P] = B\left[PBP + (I-P)\right]^{-1} B - B(I-P)B,
\]
from which it is clear that $\wt{B}$ is smooth (in particular, continuous) as a function $P \mapsto \wt{B}[P]$ on the set of density matrices.
\end{remark}

Our consideration of adaptive compression is motivated by the following fact:

\begin{proposition}[Axiomatic characterization of adaptive compression, I]
Let $B\in\mathbf{H}_{N}$ be negative definite, and let $P$ be a rank-$n$ orthogonal projector. Then $\wt{B}[P]$ is the
unique Hermitian matrix $B'$ satisfying $B'\vert_{\mathrm{Im}(P)} \equiv B\vert_{\mathrm{Im}(P)}$ and $\rank B' \leq n$. (In
fact, $\mathrm{rank}(\wt{B}[P])=n$.)
\label{prop:axiomaticI}
\end{proposition}

\begin{proof}
We have already established that $\wt{B}[P]$ satisfies the stated properties, so we need only prove uniqueness.

To this end, suppose that $B'$ is a matrix satisfying the stated properties, so $B'$ is Hermitian, has rank at most $n$, and agrees with $B$ on $\mathrm{Im}(P)$. As in the preceding discussion, let $v_1,\ldots,v_N$ be an orthonormal basis for $\mathbb{C}^N$, with $v_1,\ldots,v_n$ forming an orthonormal basis for $\mathrm{Im}(P)$. With $V_N$ as in \eqref{eqn:subspaceBasisMatrix}, write the matrix of $B'$ in this basis:
\[
[B']_{V_N} = \left(\begin{array}{cc}
B_{11}' & B_{12}'\\
B_{12}'^{*} & B_{22}'
\end{array}
\right).
\]
where the upper-left block is $n\times n$. Since $B'$ must agree with $B$ on $v_1,\ldots,v_n$, we must have $B_{11}' = B_{11}$ and $B_{12}' = B_{12}$, where the $B_{ij}$ are as in \eqref{eqn:Bblock}. In summary,
\begin{equation}\label{eqn:BprimeBlock}
[B']_{V_N} = \left(\begin{array}{cc}
B_{11} & B_{12}\\
B_{12}^{*} & *
\end{array}
\right).
\end{equation}

Since $B_{11} = V^* B V$ is invertible (where $V=[v_1, \ldots,v_n]$),
the first $n$ columns of $[B']_{V_N}$ must be linearly independent. This
means that the rank of $B'$ is at least $n$, hence equal to $n$. Then
for any $j=1,\ldots,N-n$, the $(n+j)$-th column of $[B']_{V_N}$ must be
a linear combination of the first $n$ columns. However, the coefficients
of this linear combination are completely determined by $B$, since by
\eqref{eqn:BprimeBlock} the $j$-th column of $B_{12}$ is a linear
combination of the columns of $B_{11}$ with \textit{these same
coefficients}. By the linear independence of the columns of $B_{11}$,
there is exactly one way to write each column of $B_{12}$ as a linear
combination of columns of $B_{11}$, i.e. Eq.~\eqref{eqn:tildeBblock}.
\end{proof}

\begin{remark}
For Proposition \ref{prop:axiomaticI} (and indeed for the entire discussion of Section \ref{sec:properties} thus far), it is not necessary to assume that $B \prec 0$. In fact, it is sufficient to assume that $B$ is Hermitian and $V^* B V$ is invertible. (Note that there exist invertible Hermitian matrices such that $V^* B V$ is not invertible, though this cannot happen if $B$ is definite.)  However, the case of definite $B$ affords adaptive compression with additional properties (see Lemma \ref{lem:negdef}) that are crucial for the utility of adaptive compression in solving eigenvalue problems. As discussed in Section \ref{sec:definiteReduction}, when $B$ is indefinite, the appropriate generalization of adaptive compression for the purpose of solving the eigenvalue problem \eqref{eqn:lineig} does not involve performing adaptive compression on $B$ directly, but rather reduces to the case of definite $B$ by subtracting a multiple of the identity.
\end{remark}

Before proceeding, we state a linear-algebraic result on Schur complements that will be useful for understanding the adaptive compression.
\begin{lemma}
\label{lem:Schurfact}
The positive semidefiniteness \textnormal{[}resp., definiteness\textnormal{]} of a Hermitian matrix $M:= \left(\begin{array}{cc}
X & Y\\
Y^{*} & Z
\end{array}\right)$ (where $X$ is invertible) is equivalent to having both $X\succeq0$ and $S:=Z-Y^{*}X^{-1}Y\succeq0$ \textnormal{[}resp., $\succ 0$\textnormal{]}. In fact, if $M \succeq t \geq 0$, then $S \succeq t$ as well.
\end{lemma}
\begin{remark}
Note that $S$ is a Schur complement. The first statement of Lemma \ref{lem:Schurfact} is a standard result in linear algebra (see, e.g., Theorem 1.12 of \cite{ZhangSchur}). The last statement is less widely-known, so we include an elementary proof here for completeness.
\end{remark}
\begin{proof}
We only prove the last statement. Assume that $M \succeq t$. Define
\[
F(u,v) := 
\left(\begin{array}{c}
u \\
v
\end{array}\right)^*
M
\left(\begin{array}{c}
u \\
v
\end{array}\right) = u^* X u + u^* Y v + v^* Y^* u + v^* Z v.
\]
Observe that 
\[
F( -X^{-1} Y v, v) = v^* Y^* X^{-1} Y v - v^* Y^* X^{-1} Y v - v^* Y^* X^{-1} Y v + v^* Z v = v^* S v.
\]
Using the previous two equalities and the fact that $M \succeq t \geq 0$, observe that for any $v$,
\[
v^* S v = \left(\begin{array}{c}
-X^{-1} Y v \\
v
\end{array}\right)^*
M
\left(\begin{array}{c}
-X^{-1} Y v \\
v
\end{array}\right) \geq t \left \Vert \left(\begin{array}{c}
-X^{-1} Y v \\
v
\end{array}\right) \right \Vert_2^2 \geq t \Vert v\Vert_2^2. 
\]
This completes the proof via the Courant-Fischer minimax theorem~\cite{GolubVan2013}.
\end{proof}

Taking $B$ to be negative definite, it follows from \eqref{eqn:Bproj}
that $\wt{B}[P]$ is negative semidefinite, i.e., $\wt{B}[P] \preceq 0$.
Since $\wt{B}[P]$ is a low-rank substitute for the negative definite
matrix $B$, one might additionally hope that the compression does not
make $B$ ``more negative'' in any direction, i.e. $\wt{B}[V]
\succeq B$. Lemma~\ref{lem:negdef} shows that this is indeed the case.

\begin{lemma}\label{lem:negdef}
  Let $B\in\CC^{N\times N}$ be a negative definite matrix. For any rank-$n$ projector $P$, the matrix $\wt{B}[P]-B$ is positive semidefinite. Therefore $B \preceq \wt{B}[P] \preceq 0$.
\end{lemma}

\begin{proof}
Note from \eqref{eqn:Bblock} and \eqref{eqn:tildeBblock} that 

\begin{equation}\label{eqn:differenceSchur}
\left[\wt{B}[V] - B\right]_{V_N} = \left(\begin{array}{cc}
0 & 0\\
0 & B_{12}^{*}B_{11}^{-1}B_{12} - B_{22}
\end{array}\right),
\end{equation}
so Lemma \ref{lem:negdef} is equivalent to the statement that $B_{12}^{*}B_{11}^{-1}B_{12} - B_{22} \succeq 0$, i.e., that the Schur complement $B_{22}-B_{12}^{*}B_{11}^{-1}B_{12}$ is negative semidefinite. But this follows from Lemma \ref{lem:Schurfact}, together with the fact that $B$ is negative definite.
\end{proof}

\begin{remark}
Observe that Theorem \ref{thm:main1} follows directly from Proposition \ref{prop:axiomaticI} and \ref{lem:negdef}.
Lemma \ref{lem:negdef} will also be key for proving the convergence of Algorithm \ref{alg:acelineig}.
\end{remark}

The Schur complement perspective on adaptive compression yields further
insights. Note that the stipulation that $\wt{B}[V]$ agrees with $B$ on
$\mathrm{span}~V$ determines the upper-left and upper-right blocks of
$\wt{B}[V]$ as in \eqref{eqn:tildeBblock}, and the stipulation that
$\wt{B}[V]$ is Hermitian then fixes the lower-left block. The only thing
that then remains to be specified is the lower-right block, which is
identified as in $\wt{B}[V]$.
This suggests the following characterization of adaptive compression:

\begin{proposition}[Axiomatic characterization of adaptive compression, II]
Let $B\in\CC^{N\times N}$ be a negative definite matrix, and let $E$ be an $n$-dimensional subspace of $\mathbb{C}^{N}$. Then $\wt{B}[E]$ is the
maximal Hermitian negative semidefinite matrix $B'$ satisfying $B'\vert_{E} \equiv B\vert_{E}$, in the sense that for any other such $B'$, we have $B' \preceq \wt{B}[E]$.
\label{prop:axiomaticII}
\end{proposition}

\begin{proof}
Suppose that $B' \preceq 0$ with $B' \preceq \wt{B}[E]$. Let
$V_N=[v_1,\ldots,v_N]$ be an orthonormal basis for $\mathbb{C}^N$, with $v_1,\ldots,v_n$ forming an orthonormal basis for $E$. As in the proof of Proposition \ref{prop:axiomaticI}, the matrix of $B'$ in this basis is of the form
\[
[B']_{V_N} = \left(\begin{array}{cc}
B_{11} & B_{12}\\
B_{12}^{*} & Z
\end{array}
\right),
\]
where the $B_{ij}$ are as in \eqref{eqn:Bblock}. Since $B'$ is negative semidefinite, by Lemma \ref{lem:Schurfact} the Schur complement $Z - B_{12}^* B_{11}^{-1} B_{12}$ is also negative semidefinite, i.e., $Z \preceq B_{12}^* B_{11}^{-1} B_{12}$. But by \eqref{eqn:tildeBblock} this implies that $B' \preceq \wt{B}[E]$.
\end{proof}

Similar results hold for the $t$-shifted adaptive compression.
For $B$ Hermitian, $t > \lambda_{\max} (B)$, and a rank-$n$ projector
$P$, define $\wt{B}[P,t] = \wt{B_t}[P] + t$, where $B_t = B - t$. Then we have
\begin{enumerate}
\item $B$ and $\wt{B}[P,t]$ agree on the image of $P$ denoted by $\mathrm{Im}(P)$.
\item $B \preceq \wt{B}[P,t] \preceq t$.
\item $\wt{B}[P,t]$ is not of rank $n$, but $\wt{B}[P,t]$ is the sum of
  a rank-$n$ matrix and a multiple of the identity, and hence is
  computationally efficient to apply. 
\end{enumerate}

\section{Local convergence analysis}\label{sec:local}

Since each step of \eqref{eqn:acelineig} is an Hermitian eigenvalue
problem, we can require $V^{(k)}$ to be columns of a unitary matrix and
let $P^{(k)} = V^{(k)} (V^{(k)})^*$. Then let $P^{(k+1)}$ be the density
matrix associated with $A+\wt{B}[P^{(k)}]$.  The convergence of the
adaptive compression method for the linear problem~\eqref{eqn:lineig}
can be stated in terms of the convergence of the density matrix
$P^{(k)}\to P$.
For any $H\in \mathbf{H}_{N}$, let $\lambda_i\left\{ H\right\}$ denote
its $i$-th smallest eigenvalue (counting multiplicity). In this
notation, the true eigenvalues of $(A+B)$ are $\lambda_i \equiv \lambda_i\left\{A+B \right\}$.

We now formally define the \textit{fixed point iteration map}
$\mf{F}(\cdot)$ as follows.

\begin{definition}\label{def:IterOp}
For a density matrix $Q$, let $\mf{F} (Q)$ be the orthogonal projector
$\sum_{i=1}^n u_i u_i ^*$, where the $u_i$ are orthonormal
eigenfunctions of 
\[
\big(A+\wt{B}[Q]\big) u_i = \lambda_i \big\{A+\wt{B}[Q]\big\} u_i.
\]
$\mf{F}(Q)$ is canonically defined if $A+\wt{B}[Q]$ has a positive spectral gap, such a
projector is unique, and it is the density matrix associated with
$A+\wt{B}[Q]$. Otherwise, we make an arbitrary choice in the eigenspace
associated with $\lambda_{n}\big\{A+\wt{B}[Q]\big\}$ so that
$\mf{F}(Q)$ is of rank $n$.
\end{definition}

Using the fixed point iteration map, we can rephrase Algorithm~\ref{alg:acelineig} compactly as
\begin{equation}
	P^{(k+1)}:=\mf{F} (P^{(k)} ).
  \label{eqn:fixedPt}
\end{equation}
We will see below that for $Q$ sufficiently close to $P$ (the true
density matrix of $A+B$), $A+\wt{B}[Q]$ has a positive spectral gap, and
hence its density matrix is indeed canonically defined. Thus, for all $Q$ in a
neighborhood of $P$, $\mf{F}(Q)$ is the density matrix associated with
$A+\wt{B}[Q]$. 
The local convergence of Algorithm \ref{alg:acelineig} can be studied
via the properties of the map $\mf{F}$ near the true density matrix $P$. 
A necessary requirement for Algorithm~\ref{alg:acelineig} to converge is the consistency condition
\begin{equation}
  \mf{F} (P) = P.
  \label{eqn:consistency}
\end{equation}
In order to guarantee local linear convergence, the spectral radius for
the Jacobian of $\mf{F}$ must also be bounded by unity, so that the
fixed point $P$ is attractive with respect to the
iteration~\eqref{eqn:fixedPt}.  
This leads to a sharp estimate of the
local convergence rate, which is upper-bounded by the rate provided in Theorem~\ref{thm:main2}.

\subsection{Consistency}

For a general Hermitian $B\in\mathbf{H}_{N}$, even if $\wt{B}$ is
constructed from the true density matrix $P$, the true eigenvectors
$\{v_1,\ldots,v_n\}$ of $A+B$ may not correspond to the \textit{lowest}
$n$ eigenvalues of the modified operator $A+\wt{B}[P]$, despite the fact
that they are still eigenvectors of the modified operator. In such case,
the consistency requirement~\eqref{eqn:consistency} is violated.

However, the consistency condition of the fixed point iteration will be satisfied when $B\prec 0$. 
Lemma \ref{lem:negdef} implies that $\wt{B}[P]-B \succeq 0$. Thus 
replacing $B$ by $\wt{B}[P]$ just means adding a positive
semidefinite operator that is zero on $\mathrm{span}~V$. This keeps
the bottom $n$ eigenvalues intact and shift the rest of the eigenvalues
upwards.  Lemma~\ref{lem:orderzero} verifies this statement, which
implies Eq.~\eqref{eqn:consistency}.

\begin{lemma}\label{lem:orderzero}
  Let $P$ be the density matrix associated with $A+B$, and $\wt{P}$ be the
  density matrix associated with $A+\wt{B}[P]$. Then 
  \[
  \lambda_i  \big\{ A+\wt{B}[P]\big\} \geq \lambda_i = \lambda_i  \left\{ A+B \right\}
  \]
 for $i=1,\ldots,N$, with equality if $i\leq n$. Moreover, $\wt{P}=P$.
 \vspace{1mm}
\end{lemma}

\begin{proof}
  Eq.~\eqref{eqn:Bconsistency} implies that
  $\{(\lambda_{i},v_{i})\}_{i=1}^{n}$ are eigenpairs
  of $A+\wt{B}[P]$. Hence it is sufficient to show that
  $\{\lambda_{i}\}_{i=1}^{n}$ are also the lowest $n$ eigenvalues.

  The Courant-Fischer minimax theorem~\cite{GolubVan2013} and
  Lemma~\ref{lem:negdef} give
  \begin{equation}
    \begin{split}
      \lambda_i  \left\{ A+B \right\}
 &= \min_{\mathrm{dim}S=i} \max_{0\ne u\in S}
      \frac{u^{*} (A+\wt{B}[P])u}{u^{*}u} \\
      &= \min_{\mathrm{dim}S=i} \max_{0\ne u\in S}
      \left(\frac{u^{*} (A+B)u}{u^{*}u} + \frac{u^{*}
      (\wt{B}[P]-B)u}{u^{*}u}\right)\\
      &\ge \min_{\mathrm{dim}S=i} \max_{0\ne u\in S}
      \frac{u^{*} (A+B)u}{u^{*}u} = \lambda_{i}.
    \end{split}
    \label{eqn:orderzero_2}
  \end{equation}
  Since
  $\{\lambda_{i}\}_{i=1}^{n}$ are already eigenvalues, the only
  possibility is that $\lambda_{i}=\lambda_i  \left\{ A+B \right\}$ for $1\le i\le n$,
  and hence $P=\wt{P}$.
\end{proof}

We now verify that $\mf{F}$ is canonically defined for density matrices in a neighborhood of $P$. 
This amounts to proving that $A+\wt{B}[Q]$ has a spectral gap for
density matrices $Q$ sufficiently close to $P$. By Lemma
\ref{lem:orderzero}, the spectral gap of $A+\wt{B}[P]$ is at least as
large as the spectral gap of $A+B$ denoted by $\lambda_g$. In particular, the
spectral gap of $A+\wt{B}[P]$ is positive. Then since the $k$-th
eigenvalue of a Hermitian matrix $M$ is a Lipschitz function of $M$ 
(see e.g. ~\cite{GolubVan2013}),
and since $\wt{B}[Q]$ is continuous in the
density matrix $Q$ (see Remark \ref{rem:compressionContinuity}),
$A+\wt{B}[Q]$ has a positive spectral gap for density matrices $Q$
sufficiently close to $P$, as claimed.

\subsection{Linearization}\label{subsec:linearLocal}

We study the response of $\mf{F}$ to a small perturbation of $P$ in
two steps. First we determine the change in the density matrix induced
by a small perturbation of the matrix $H=A+\wt{B}[P]$. This gives a Jacobian
denoted by $D P_{H}$. 
Then we describe how $\wt{B}$
(and hence also the matrix $A+\wt{B}$ of the eigenvalue problem in
question in each iteration) responds to the small perturbation of $P$.
This gives a Jacobian $D \wt{B}_{P}$.
The composition of these Jacobian operators yields the Jacobian of $\mf{F}$
at $P$, denoted by $D \mf{F}_{P}$.
In the physics literature for solving Hartree-Fock-like
equations, $D P_{H}$ is called the \textit{irreducible polarizability matrix}.

For any orthogonal projector $Q$, let $Q^\perp := I-Q$ denote the
orthogonal projector onto $\mathrm{Im}(Q)^\perp$. We first give
explicit expressions for $D P_{H}$ and $D \wt{B}_{P}$ in
Lemma~\ref{lem:contour} and~\ref{lem:fB} respectively, for which the
proofs are given in Appendix \ref{appendix:calculations}.

\begin{lemma}\label{lem:contour}
  For $H\in\mathbf{H}_N$ with a positive spectral gap, $\Delta H\in
  \mathbf{H}_N$, $1\le n\le N$ and $\epsilon>0$ sufficiently small, let $P,P_\epsilon$
  be the rank-$n$ density matrices associated with $H$ and
  $H+\epsilon \Delta H$, respectively.
  Then
  \begin{eqnarray*}
    D P_H[\Delta H] & = & \sum_{i=1}^{n}\sum_{a=n+1}^{N} \frac{1}{\mu_{i}-\mu_{a}}
    u_{a} (u_{a}^{*} \Delta H u_{i}) u_{i}^{*} + \mathrm{h.c.}
    \label{eqn:contour_1b} \\
    & = & \sum_{i=1}^{n} \left[P^\perp (\mu_i - H)P^\perp \right]^\dagger \Delta H u_{i} u_{i}^{*} + \mathrm{h.c.},
  \end{eqnarray*}
  where $\mathrm{h.c.}$ stands for the Hermitian conjugate of the term
  that precedes it and $u_1 ,\ldots, u_N$ are orthonormal eigenvectors
  of $H$ with corresponding eigenvalues $\mu_1 \leq \cdots \leq \mu_N$.
  (Note that $\mu_n < \mu_{n+1}$ by assumption.)
\end{lemma}

\begin{lemma}\label{lem:fB}
  For $\epsilon>0$ sufficiently small, let $P,P_\epsilon$ be density matrices with $\Delta P = \lim_{\epsilon\to
  0} (P_\epsilon- P)/\epsilon$. Then
  \begin{equation}
    D \wt{B}_P[\Delta P]:= \lim_{\epsilon\to 0} \frac{\wt{B}[P_\epsilon]-\wt{B}[P]}{\epsilon} =
     \left(B-\wt{B}[P]\right) (\Delta P)
    (PBP)^\dagger B + \mathrm{h.c.}
    \label{eqn:fB_1}
  \end{equation}
\end{lemma}

The composition of Lemma \ref{lem:contour} with Lemma \ref{lem:fB} gives
an explicit expression for $D \mf{F}_{P}$:

\begin{lemma}\label{lem:composedJacobian}
 For $\epsilon>0$ sufficiently small, let $P,P_\epsilon$ be density matrices with $\Delta P = \lim_{\epsilon\to
  0} (P_\epsilon- P)/\epsilon$. Then
  \begin{eqnarray*}
    D\mf{F}_P[\Delta P] &:=& \lim_{\epsilon\to 0} \frac{\mf{F}(P_\epsilon)-\mf{F}(P)}{\epsilon} \\
    &=& 
    \sum_{i=1}^n
    \left(P^\perp + \left(\wt{B}[P]-B\right)^\dagger \left(A+B- \lambda_i\right) P^\perp \right)^\dagger    
    (\Delta P) v_i v_i^* + \mathrm{h.c.}
  \end{eqnarray*}
\end{lemma}

\begin{proof}
Applying Lemma \ref{lem:contour} (with  $H=A+\wt{B}[P]$ and $\Delta H = D \wt{B}_P [\Delta P]$) and Lemma \ref{lem:fB}, we have
\begin{eqnarray*}
    D\mf{F}_P[\Delta P] &=&  \sum_{i=1}^{n} \left[P^\perp(\lambda_i - H) P^\perp \right]^\dagger \left[ \left(B - \wt{B}[P] \right) (\Delta P) (P B P)^\dagger B + \mathrm{h.c.} \right]  v_{i} v_{i}^{*} + \mathrm{h.c.}
\end{eqnarray*}
For $i=1,\ldots,n$, $\big(B - \wt{B}[P] \big)v_i = 0$ and $(P B P)^\dagger B v_i = B^{-1} \wt{B}[P] v_i = v_i$, so our expression for $D\mf{F}_P[\Delta P]$ simplifies to 
\begin{eqnarray*}
D\mf{F}_P[\Delta P] & =& \sum_{i=1}^{n} \left[P^\perp \left(\lambda_i - A - \wt{B}[P] \right) P^\perp \right]^\dagger  \left(B - \wt{B}[P] \right) (\Delta P) v_{i} v_{i}^{*} + \mathrm{h.c.} \\
& = &  \sum_{i=1}^{n} \left[P^\perp \left(\wt{B}[P] - B + A + B - \lambda_i \right) P^\perp \right]^\dagger  \left(\wt{B}[P] - B\right) (\Delta P) v_{i} v_{i}^{*} + \mathrm{h.c.}
\end{eqnarray*}
Now for any $i=1,\ldots,n$, $\mathrm{Im}(P)$ is an invariant subspace
for the self-adjoint operator $A + \wt{B}[P] - \lambda_i$, and
$\mathrm{Im}(P)^\perp$ is an invariant subspace as well. As an operator
$\mathrm{Im}(P)^\perp \rightarrow \mathrm{Im}(P)^\perp$, $A + \wt{B}[P]
- \lambda_i$ is positive definite and hence invertible. Thus the
pseudoinverse in the preceding expression is 
effectively taking a matrix
inverse on the lower-right block the matrix representation as
in~\eqref{eqn:Bblock}, while all other blocks are zero.

Similarly, $\mathrm{Im}(P)^\perp$ is invariant for $\wt{B}[P] - B$,
which is only nonzero in its lower-right block. By Lemma
\ref{lem:Schurfact}, $\wt{B}[P] - B$ is positive definite (hence
invertible) as an operator $\mathrm{Im}(P)^\perp \rightarrow
\mathrm{Im}(P)^\perp$. By taking the factor of $\wt{B}[P] - B$ inside of
the pseudoinverse we obtain the desired equality.  \end{proof}

For $i=1,\ldots,n$, define 
\[
Z_i := \left(P^\perp + \left(\wt{B}[P]-B\right)^\dagger \left(A+B- \lambda_i\right) P^\perp \right)^\dagger.
\]
then Lemma \ref{lem:composedJacobian} can be equivalently expressed as
\[
D\mf{F}_P[\Delta P] = \sum_{i=1}^{n} Z_i (\Delta P) v_{i} v_{i}^{*} + \mathrm{h.c.}
\]

\begin{remark}\label{remark:Zmatrix}
The matrix of the linear transformation $Z_i$ in Lemma \ref{lem:composedJacobian} is given by 
\[
\left[ Z_i \right]_{V_N} =
\left(\begin{array}{cc}
0 & 0\\
0 & J_i
\end{array}\right),
\]
where 
\[
J_i := \left[ I_{N-n} - \left( S_{22} \right)^{-1} (\Lambda_2 -
\lambda_i) \right]^{-1}.
\]
Here $\Lambda_2 := \mathrm{diag}(\lambda_{n+1},\ldots,\lambda_N)$, and
\[
S_{22}=B_{22} - B_{12}^* B_{11}^{-1} B_{12}
\]
is the Schur complement with $S_{22} \prec 0$. 
\end{remark}

We can view the Jacobian $D\mf{F}_P$ as a linear operator on the tangent
space at $P$ of the manifold of all rank-$n$ density matrices. 
We will see later that the set of eigenvalues of $D\mf{F}_P$ is the
union of the set of eigenvalues of $\{J_i\}$. We find an upper bound
for all eigenvalues of $\{J_{i}\}$ in Lemma~\ref{lem:eigenJ}:

\begin{lemma}\label{lem:eigenJ}
For $i=1,\ldots,n$, $J_i$ is diagonalizable with $\sigma(J_i) \subset (0,1)$ and
\begin{eqnarray*}
\gamma := 
\max_{i=1,\ldots,n} \lambda_{\max} (J_i)
 \leq \frac{\Vert S_{22} \Vert_2}{\lambda_g + \Vert S_{22} \Vert_2 } 
\leq \frac{\Vert P^\perp B P^\perp \Vert_2}{\lambda_g + \Vert P^\perp B P^\perp \Vert_2 } \leq \frac{\Vert B \Vert_2}{\lambda_g + \Vert B \Vert_2 } < 1.
\end{eqnarray*}
\end{lemma}
\begin{proof}
We adopt the notation used in Remark \ref{remark:Zmatrix}. Since the
eigenvalues of a matrix are invariant under conjugation (i.e. similarity
transformation), conjugating $J_i$ by $(\Lambda_2 - \lambda_i)^{1/2}$
yields the equality of spectra
\[
\sigma \left( J_i \right)
= \sigma \left( \left[ I_{N-n} + (\Lambda_2 - \lambda_i)^{1/2} \left( -S_{22} \right)^{-1} (\Lambda_2 - \lambda_i)^{1/2} \right]^{-1} \right).
\]
Here the equality is defined in the sense of sets.
The matrix on the right-hand side is positive definite, so $\sigma(Z_i) \subset (0,1)$ as claimed. In fact, the matrix $(\Lambda_2 - \lambda_i)^{1/2} \left( -S_{22} \right)^{-1} (\Lambda_2 - \lambda_i)^{1/2}$ is positive definite and we have 
\[
\sigma \left( J_i \right)
= \frac{1}{1 + \sigma \left[  (\Lambda_2 - \lambda_i)^{1/2} \left( -S_{22} \right)^{-1} (\Lambda_2 - \lambda_i)^{1/2} \right] }.
\]

Now observe
\begin{eqnarray*}
&& \lambda_{\min} \left[ (\Lambda_2 - \lambda_i)^{1/2} \left( -S_{22} \right)^{-1} (\Lambda_2 - \lambda_i)^{1/2} \right] \\
\qquad =&&  \left( \lambda_{\max} \left[ (\Lambda_2 - \lambda_i)^{-1/2} \left( -S_{22} \right) (\Lambda_2 - \lambda_i)^{-1/2} \right] \right) ^{-1}\\
 \qquad \geq&& \left( \Vert (\Lambda_2 - \lambda_i)^{-1/2} \Vert_2 ^2 \cdot  \Vert S_{22} \Vert_2 \right)^{-1} \\
 \qquad \geq&& \frac{\lambda_g}{\Vert S_{22} \Vert_2}.
\end{eqnarray*}

This establishes the first claimed inequality. Recall that $ S_{22} \preceq 0$, but also $ S_{22} = B_{22} - B_{12}^* B_{11}^{-1} B_{12}$, so $S_{22} \succeq B_{22}$. Thus
\[
\Vert S_{22} \Vert_2 \leq \Vert B_{22} \Vert_2 = \Vert P^\perp B P^\perp \Vert_2 \leq \Vert B \Vert _2.
\]
Since $x\mapsto x/(1+x)$ is increasing for $x\geq 0$, this proves the rest of the inequalities.

The diagonalizability of $J_i$  is implied by the similarity
transformation.
\end{proof}

\subsection{Dynamical systems perspective on adaptive compression}

In order to study the local convergence properties of the fixed point iteration map $\mf{F}$, we first note that the set of all density matrices $\mathcal{D}$ is not a subspace, but a smooth submanifold of
$\mathbb{C}^{N\times N} \simeq \mathbb{R}^{2N^2}$. 
$\mathcal{D}$ can be
identified with the Grassmannian
$\mathbf{Gr}(n,\mathbb{C}^N)$, which is the set of all complex
$n$-dimensional subspaces of $\mathbb{C}^N$. 
Since the fixed point iteration map $\mf{F}$ is a map from $\mathcal{D}$
to itself and is smooth on
a neighborhood of $P$, we consider
the linearization of $\mf{F}$ about the fixed point $P$ is the tangent space
$T_P \mathcal{D} \subset \mathbb{C}^{N\times N} \simeq
\mathbb{R}^{2N^2}$. This tangent space can be characterized as follows.

First note that any smooth path of rank-$n$ density matrices, denoted by $\gamma(t)$ with
$\gamma(0) = P$, can be expressed as
\[
\gamma(t) = V_N U(t) \left(\begin{array}{cc}
I & 0 \\
0 & 0
\end{array}\right) U(t)^* V_N^*,
\]
where $U(t)$ is a smooth path of unitary matrices with $U(0) = I$ and $V_N = [v_1,\ldots,v_N]$. Since the Lie algebra of the unitary group (i.e. the tangent space at the identity element) is the set of skew-Hermitian matrices, we have 
\[
U'(0) = \left(\begin{array}{cc}
Y & -X^* \\
X & Z
\end{array}\right),
\]
where $Y^* = -Y$ and $Z^* = -Z$. Then
\[
\gamma'(0) = V_N \left(\begin{array}{cc}
Y & -X^* \\
X & Z
\end{array}\right)
\left(\begin{array}{cc}
I & 0 \\
0 & 0
\end{array}\right) V_N^* + \mathrm{h.c.} = \widetilde{V} \left(\begin{array}{cc}
0 & X^* \\
X & 0
\end{array}\right)
V_N^*.
\]
Hence the tangent space
\begin{equation}\label{eqn:tangentSpace}
T_P \mathcal{D} = \left\{ V_N \left(\begin{array}{cc}
0 & X^* \\
X & 0
\end{array}\right) V_N^* \,:\,X\in\mathbb{C}^{(N-n)\times n}  \right\},
\end{equation}
and we can make the identification $T_P \mathcal{D} \simeq \mathbb{C}^{(N-n)\times n}$.
Observe that the map $\Phi : \mathbb{C}^{(N-n)\times n} \rightarrow \mathcal{D}$ defined by 
\begin{equation}\label{eqn:tangentDiffeo}
X \mapsto V_N 
\left[\exp \left(\begin{array}{cc}
0 & -X^* \\
X & 0
\end{array}\right) \right]
\left(\begin{array}{cc}
I & 0 \\
0 & 0
\end{array}\right)
\left[ \exp \left(\begin{array}{cc}
0 & -X^* \\
X & 0
\end{array}\right) \right]^*
V_N^*
\end{equation}
is a local diffeomorphism near the origin. Then for $Q\in\mathcal{D}$
sufficiently close to $P$, we can identify $Q$ with $X:=\Phi^{-1}(Q)\in
\mathbb{C}^{(N-n)\times n}$. Then we can identify $\mf{F}$ with a map
$\mf{H}$ defined on a neighborhood $\mathcal{U}$ of the origin in
$\mathbb{C}^{(N-n)\times n}$

\begin{remark}
Adopting this perspective, Remark \ref{remark:Zmatrix} translates to 
\[
D\mf{H}_0 [X] = (J_1 X_1, \ldots, J_n X_n),
\]
for any $X = (X_1, \ldots, X_n) \in \mathbb{C}^{(N-n)\times n}$, where
$D\mf{H}_0$ is the usual Jacobian of the map $\mathcal{U}
\rightarrow \mathbb{C}^{(N-n)\times n}$ at the origin, naturally viewed
as a tensor in $\mathbb{C}^{(N-n)\times n \times (N-n) \times n}$.
Identifying tangent vector $X$ with its vectorization in
$\mathbb{C}^{(N-n)n}$, the matricized representation of $D\mf{H}_0$
in $\mathbb{C}^{(N-n)n\times (N-n)n}$ yields
\[
D\mf{H}_0 = 
\left(\begin{array}{cccc}
J_1 & 0 & \cdots  &0 \\
0 & J_2 & \cdots & 0 \\
\vdots & \vdots & \ddots &\vdots \\
0  &0 &\cdots  & J_n
\end{array}\right).
\]
\end{remark}

Near the fixed point $P$, we can view Algorithm \ref{alg:acelineig} as a discrete-time dynamical system on $\mathbb{C}^{(N-n)n}$. The stability of the fixed point $P$ is then determined by the spectrum $\sigma(D\mf{H}_0)$ of the Jacobian $D\mf{H}_0$, which is the union of the spectra $\sigma(J_i)$ over $i=1,\ldots,n$.

\subsection{Asymptotic convergence rate}

We will make use the following Lemma to show the local convergence.

\begin{lemma}\label{lem:stableManifoldFullDim}
Let $F:\mathbb{R}^p \cap B_{\delta}(0) \rightarrow \mathbb{R}^p$ be a
smooth map such that $F(0)=0$, $DF(0)$
is diagonalizable, and the spectral radius $\gamma := \sup \vert \sigma(DF(0)) \vert$ of $DF(0)$ is strictly less than $1$. Then for any $\epsilon >0$ for which $\gamma + \epsilon < 1$, there exist constants $C,c>0$ such that if $\Vert x\Vert_2 < c$, then $\Vert F^k (x) \Vert_2 \leq C (\gamma + \epsilon)^k \Vert x\Vert_2$ for all $k\geq 0$.
\end{lemma}
\begin{proof}
First note that we can assume that in fact $DF(0)$ is diagonal by replacing $F$ with $\phi^{-1} \circ F \circ \phi$ for a suitable change of basis $\phi$. Then $\Vert DF(0) \Vert_2 = \gamma$, and there exists $c$ such that $\Vert y\Vert_2 < c$ implies $\Vert DF(y) \Vert_2 < \gamma + \epsilon$. Thus if $\Vert x\Vert_2 <c$, then
\begin{eqnarray*}
\Vert F(x)\Vert_2 &=& \Vert F(x) - F(0) \Vert_2 \\
&=& \left\Vert \int_{0}^1 DF(tx)\cdot x \,\ud t \right\Vert_2 \\
&\leq & \int_{0}^1 \Vert DF(tx)\Vert_2 \Vert x\Vert_2 \,\ud t \leq (\gamma + \epsilon) \Vert x\Vert_2.
\end{eqnarray*}
Repeated application of this inequality yields the result.
\end{proof}

\begin{remark}
The reader familiar with dynamical systems should note that Lemma
\ref{lem:stableManifoldFullDim} is almost a recapitulation of the stable
manifold theorem in the case that the local stable manifold has full
dimension.
\end{remark}

Now we are ready to prove Theorem~\ref{thm:main2}, which is stated more
precisely in Theorem~\ref{thm:convlineig}.

\begin{theorem}\label{thm:convlineig}
Let $\epsilon > 0$ be small enough so that $\gamma + \epsilon < 1$, where $\gamma$ is as in Lemma \ref{lem:eigenJ}. 
Then there exist constants $C,c>0$ such that if $\Vert P^{(0)} - P \Vert_2 \leq c$, then 
\[
\Vert P^{(k)}-P \Vert_2 \leq C (\gamma+\epsilon)^k \norm{P^{(0)} -
P}_{2}
\]
for all $k\geq 0$.
\end{theorem}
\begin{proof}
Fix $\epsilon$ as in the statement of the theorem. We can identify
$\mathbb{C}^{(N-n)n}$ with $\mathbb{R}^{2(N-n)n}$, and the corresponding
realification of $D\mf{H}_0$ has all of its eigenvalues in
$(0,\gamma]$. (It has two copies of each of the eigenvalues of $D\mf{H}_0$ as an operator $\mathbb{C}^{(N-n)n}\rightarrow \mathbb{C}^{(N-n)n}$.) By Lemma \ref{lem:stableManifoldFullDim}, there exists a neighborhood $\mathcal{V}$ of 0 within $\mathcal{U} \subset \mathbb{C}^{(N-n)n}$ and a constant $C$ such that $\mf{H}^k (X^{(0)}) \in \mathcal{V} $ and moreover $\Vert \mf{H}^k (X^{(0)})\Vert_2 \leq C (\gamma + \epsilon)^k \Vert X^{(0)} \Vert_2 $ for all $k\geq 0$. From \eqref{eqn:tangentDiffeo} we have
\[
P^{(k)} = 
\widetilde{V}
\left[\exp \left(\begin{array}{cc}
0 & -(X^{(k)})^* \\
X^{(k)} & 0
\end{array}\right) \right]
\left(\begin{array}{cc}
I & 0 \\
0 & 0
\end{array}\right)
\left[\exp \left(\begin{array}{cc}
0 & -(X^{(k)})^* \\
X^{(k)} & 0
\end{array}\right) \right]^*
\widetilde{V}^*.
\]
Since
\[
P = 
\widetilde{V}
\left(\begin{array}{cc}
I & 0 \\
0 & 0
\end{array}\right)
\widetilde{V}^*,
\]
it follows (for a possibly enlarged constant $C$) that $\Vert P^{(k)}-P \Vert_2 \leq C (\gamma+\epsilon)^k \Vert P^{(0)} - P \Vert_2$, as was to be shown.
\end{proof}

\begin{remark}
Recall from Lemma \ref{lem:eigenJ} that $\gamma \leq \frac{\Vert B \Vert_2}{\lambda_g + \Vert B \Vert_2 } < 1$, so we have a linear rate of convergence that depends only on the ratio $\Vert B \Vert_2 / \lambda_g$. If this ratio is smaller, then the convergence is faster, and vice-versa.
\end{remark}

\subsection{Convergence of sub-projectors}
Now we prove Theorem \ref{thm:sub-projector} regarding the rate of convergence of the rank-$m$ sub-projectors $P_m^{(k)}$ to $P_m$. In this section we use $C$ to denote a constant that possibly changes across usages and is understood to be sufficiently large in each context.

The important observation is that $P_m$ can be identified with an invariant submanifold for the dynamics, to which the dynamics are attracted via a (relatively) rapid transient.

Consider
\[
\mathcal{D}_m := \{Q\in \mathcal{D}\,:\,Q\vert_{\mathrm{Im}(P_{m})}=\mathrm{Id}_{\mathrm{Im}(P_{m})}\},
\]
which is a submanifold of $\mathcal{D}$, and can be identified as the submanifold of
$\mathbf{Gr}(n,\mathbb{C}^N)$ consisting of the $n$-dimensional
subspaces of $\mathbb{C}^N$ that contain $\mathrm{Im}(P_{m})$. This is in turn
isomorphic to $\mathbf{Gr}(n-m,\mathbb{C}^N / \mathrm{Im}(P_{m})) \simeq
\mathbf{Gr}(n-m,\mathbb{C}^{N-m })$.

We assume that $\lambda_{m+1}-\lambda_m > 0$, and we allow this gap to
be small in practice. Then
\begin{equation*}\label{eqn:QcloseP}
  \lambda_i\{A+\wt{B}[Q]\} = \lambda_i\{A+\wt{B}[P]\} + \Or(\Vert
Q-P\Vert_2^2).
\end{equation*}
In particular, for each $i=m+1,\ldots,N$, we have
$\lambda_i\{A+\wt{B}[Q]\} > \lambda_{m}$ for all $Q$ sufficiently close
to $P$. If $Q\in \mathcal{D}_m$, then $(\lambda_i,v_i)$
is an eigenpair for $A+\wt{B}[Q]$ for $i=1,\ldots,m$, and these
eigenvalues are the lowest $m$ eigenvalues of $A+\wt{B}[Q]$. It follows that $\mathrm{Im}(\mf{F}(Q)) \supset \mathrm{Im}(P_{m})$. Hence
near the fixed point $P$, $\mathcal{D}_m$ is invariant under the fixed point iteration map $\mf{F}$.

For $Q\in \mathcal{D}$, define $\mf{F}_m(Q)$ to be the rank-$m$
projector onto the span of the lowest $m$ eigenvectors of $A+\wt{B}[Q]$.
The assumption $\lambda_{m+1}-\lambda_m > 0$ guarantees that this map is
canonically defined and smooth near $P$, and $\mf{F}_m (Q) = P_m$ for
all $Q\in \mathcal{D}_m$ sufficiently close to $P$.  Then there is a
neighborhood $\mathcal{N}$ of $P$ in $\mathcal{D}$ such that
$\mf{F}(\mathcal{D}_m\cap\mathcal{N}) \subset \mathcal{D}_m$ and such
that $\mf{F}_m (\mathcal{D}_m \cap\mathcal{N}) = P_m$. In particular, we
have constructed a local invariant manifold $\mathcal{D}_m$ for the
dynamics due to the fixed point iteration.

We would like to prove that
the dynamics converge rapidly to this invariant manifold locally, in the
sense that
\begin{equation}\label{eqn:subprojClaim}
\mathrm{dist}(P^{(k)},\mathcal{D}^m\cap\mathcal{N}) \leq C \gamma_m^k \cdot
\mathrm{dist}(P^{(0)},\mathcal{D}^m\cap\mathcal{N}),
\end{equation}
where we can take $\gamma_m := \Vert B\Vert_2 / (\Vert B
\Vert_2 + \Delta)$ and where `$\mathrm{dist}$' indicates the distance between sets induced by the norm $\Vert \cdot \Vert_2$. 
We claim that in fact Theorem \ref{thm:sub-projector} would follow from
\eqref{eqn:subprojClaim}, together with the preceding remarks. We will justify the choice
of constant $\gamma_m$ later, but for now we map out the rest of the
argument.

To see the claim, note that since $\mf{F}_m$ is smooth near $P$ (hence in particular locally Lipschitz), there exists $L$ such that, for all $\epsilon$ sufficiently small, if $Q$ satisfies $\mathrm{dist}(Q,\mathcal{D}_m \cap \mathcal{N}) \leq \epsilon$, then 
\[
L \epsilon \geq \mathrm{dist}(\mf{F}_m(Q), \mf{F}_m(\mathcal{D}_m\cap\mathcal{N}) ) = \mathrm{dist}(\mf{F}_m(Q),P_m ) = \Vert \mf{F}_m(Q) - P_m \Vert_2.
\]
Thus if we can establish \eqref{eqn:subprojClaim}, then substituting $Q = P^{(k)}$ yields
\[
\Vert P^{(k+1)}_m - P_m \Vert \leq C \gamma_m^k \cdot \mathrm{dist}(P^{(0)},\mathcal{D}^m\cap\mathcal{N}) \leq C \gamma_m^k \Vert P - P^{(0)} \Vert_2,
\]
establishing Theorem \ref{thm:sub-projector}.

We have then reduced Theorem \ref{thm:sub-projector} to the following
lemma.

\begin{lemma}\label{lem:sub-projectorReduction}
There is a neighborhood $\mathcal{W}$ of $P$ in $\mathcal{D}$ such that if $P^{(0)} \in \mathcal{W}$, then $\mathrm{dist}(P^{(k)},\mathcal{D}_m\cap\mathcal{N}) \leq C \gamma_m^k \cdot \mathrm{dist}(P^{(0)},\mathcal{D}^m\cap\mathcal{N})$.
\end{lemma}

In order to motivate the constant $\gamma_m$, note that $T_P \mathcal{D}_m$, considered as a subspace of $\mathbb{C}^{(N-n)n} \simeq T_P \mathcal{D}$, is given by 
\[
T_2 := \{ (X_1,\ldots,X_n)\in \mathbb{C}^{(N-n)n}\,:\,X_1=\cdots=X_m = 0 \},
\]
and we have locally the splitting $ T_P \mathcal{D} \simeq
\mathbb{C}^{(N-n)n} = T_1 \oplus T_2$, where 
\[
T_1 := \{ (X_1,\ldots,X_n)\in \mathbb{C}^{(N-n)n}\,:\,X_{m+1}=\cdots=X_N = 0 \}.
\]
Observe that the eigenvalues of $D\mf{H}_0\vert_{T_1}$ are the eigenvalues of $J_1,\ldots,J_m$. By the proof of Lemma \ref{lem:eigenJ}, all of these eigenvalues are in $(0,\gamma_m)$, so the spectrum of $D\mf{H}_0\vert_{T_1}$ is contained in $(0,\gamma_m)$. (The eigenvalues of $D\mf{H}_0\vert_{T_2}$ are the eigenvalues of $J_{m+1},\ldots,J_N$, which are all in $(0,1)$.) At least formally, this discussion motivates the statement of Lemma \ref{lem:sub-projectorReduction}.
By considering a smooth change of coordinates near $P$ that straightens the invariant submanifold $\mathcal{D}_m$ and then diagonalizes the Jacobian, we can replace Lemma \ref{lem:sub-projectorReduction} with the following:
\begin{lemma}\label{lem:invariantSubspace}
Let $F:\mathbb{R}^p \cap B_{\delta}(0) \rightarrow \mathbb{R}^p$ be a
smooth map such that $F(0)=0$, $DF(0)$ is diagonal, $0\prec DF(0) \prec
1$, and $DF(0)\vert_{E_1} \prec \alpha < 1$, where $E_1 :=
\mathrm{span}\,\{e_1,\ldots,e_r\}$ for $r \leq p$. Further suppose that $E_2 := E_1^\perp$ is invariant under $F$, i.e.,
$F(E_2 \cap B_{\delta}(0)) \subset E_2$. Then there exists $\delta' \in (0,\delta)$ such that $F$ maps $B_{\delta'}(0)$ into itself and such that for any $x\in B_{\delta'}(0)$,
\[
\mathrm{dist}(F^k(x),E_2 \cap B_{\delta}(0)) \leq \alpha^k\cdot \mathrm{dist}(x,E_2 \cap B_\delta (0)).
\]
\end{lemma}
\begin{proof}
See Appendix \ref{appendix:sub-projector}. 
\end{proof}
\begin{remark}
Note carefully that we \textnormal{do not} consider a change of
coordinates that produces a linear dynamical system, i.e., we do not
assume $F$ is linear in Lemma \ref{lem:invariantSubspace}. In general,
such a change of coordinates does exist near a hyperbolic fixed point by
the Hartman-Grobman theorem (see, e.g., Theorem 10.4 of
\cite{TeschlODE}), but it is only guaranteed to be a homeomorphism (not
necessarily Lipschitz). We need the change of coordinates to be
Lipschitz in order to compare distances up to a constant. 
\end{remark}

\section{Global convergence analysis}\label{sec:global}

Before providing a roadmap for the proof of the global convergence
properties in Theorem \ref{thm:main3}, we first show that the adaptive compression method \textit{cannot} be expected to converge globally to the solution of~\eqref{eqn:lineig} for  \textit{every} initial guess $P^{(0)}$.

Consider taking $N=2$, $n=1$, $A$ is a zero matrix, and \[
B=
\left(\begin{array}{cc}
-2 & 0\\
0  & -1
\end{array}\right).
\]
Note that the true density matrix is $P = e_1 e_1^*$, where
$e_{1}=(1,0)^{T},e_{2}=(0,1)^{T}$. However $e_2 e_2^*$ is also a fixed
point of $\mf{F}$. Thus if we take $P^{(0)} = e_2 e_2^*$, we get
convergence to the wrong fixed point.

A slightly more sophisticated example demonstrates that it is possible for Algorithm \ref{alg:acelineig} to stall on some incorrect fixed point, even if not initialized there. Take $N=3$ and $n=1$ with
\[
A=
\left(\begin{array}{ccc}
0 & 0 & 0 \\
0  & -2 & 0\\
0 & 0 & 0
\end{array}\right),
\quad
B=
\left(\begin{array}{ccc}
-4 & 0 & 0 \\
0  & -1 & 0\\
0 & 0 & -1
\end{array}\right).
\]
The true density matrix is $P = e_1 e_1^*$. However, suppose that $P^{(0)} = e_3 e_3^*$. Then 
\[
A+\wt{B}[P^{(0)}] = \left(\begin{array}{ccc}
0 & 0 & 0 \\
0  & -2 & 0\\
0 & 0 & -1
\end{array}\right),
\]
so $P^{(1)} = e_2 e_2^*$. Now
\[
A+\wt{B}[P^{(1)}] = \left(\begin{array}{ccc}
0 & 0 & 0 \\
0  & -3 & 0\\
0 & 0 & 0
\end{array}\right),
\]
so $P^{(1)}$ is a fixed point, and Algorithm \ref{alg:acelineig} fails
to converge. 

Therefore we can only hope for convergence for \emph{almost every} choice of initial guess. In the sequel we will see that such incorrect fixed points are unstable, and this observation will allow us to prove an almost-sure convergence result.

\subsection{Outline of the proof of global convergence} 
Before embarking on the global convergence analysis, we
pause to provide a detailed outline of Section \ref{sec:global}. The reader may find it useful to refer back to this outline throughout the section.

In Section \ref{subsec:eigenMono} we introduce a key property of
Algorithm \ref{alg:acelineig}: each of the bottom $n$ eigenvalues
$\lambda_{1}^{(k)},\ldots,\lambda_n^{(k)}$ of $A+\wt{B}[P^{(k-1)}]$ is
monotonically non-increasing in $k$. We call this property
\textit{eigenvalue monotonicity}. Eigenvalue monotonicity implies that
$\sum_{i=1}^n \lambda_i^{(k)}$ is convergent in $k$. In particular,
when $k$ is large, $\sum_{i=1}^n \lambda_i^{(k)}$ does not change much
across iterations. Lemma \ref{lem:densMatChange}
shows that the change of $P^{(k)}$ across one iteration can be controlled by
the change of $\sum_{i=1}^n \lambda_i^{(k)}$. So when $k$ is large,
$\mf{F}(P^{(k)}) \approx P^{(k)}$, i.e., the point $P^{(k)}$ is almost
fixed by the mapping $\mf{F}$.

Unfortunately, this is not yet enough to directly imply that the sequence
$P^{(k)}$ is convergent, but one might hope that a point that is close
to being fixed is close to some fixed point! Notice that a fixed point
$P_f$ of $\mf{F}$ must satisfy the condition that $\mathrm{Im}(P_f)$ is
an invariant subspace for $A+B$, i.e., must satisfy $P_f =
\sum_{i=1}^{n} u_i u_i ^*$, where $u_i$ are eigenvectors of $A+B$. We
will show (see Lemma \ref{lem:almostEigen}) that a point that is almost
fixed is indeed almost a point of this form, i.e., an orthogonal projector onto an invariant subspace of $A+B$.

To avoid pathologies, one hopes that there are only finitely many projectors of this form, and indeed this is the case if $A+B$ has distinct eigenvalues. This observation brings us to Section \ref{subsec:genericity}, wherein we impose conditions on our matrices $A$ and $B$ that hold generically in a precise sense and that allow us to avoid nongeneric pathologies in the proof of global convergence. One of these conditions, as we have said, is that $A+B$ has distinct eigenvales. The other, which is more technical, guarantees that $A+\wt{B}[P_f]$ has a spectral gap at every fixed point $P_f$. This will allow us to perform linearization at fixed points (which are not necessarily the true density matrix $P$) as we have done in Section \ref{subsec:linearLocal}.

With these new assumptions at hand, we proceed with the proof of global convergence. In Section \ref{subsec:globalConvergenceToFixed}, we establish that we have global convergence to a fixed point, though we do not yet say anything about whether or not this limit point is the true density matrix $P$. The argument proceeds as follows. We now know that for every $k$ sufficiently large, $P^{(k)}$ is close to some orthogonal projector onto an invariant subspace of $A+B$ and moreover that there are only finitely many projectors $P_\tau$ of this form.
These points must be mutually isolated
since they are only finite in number. Since (1) $P^{(k)}$ must be close
to at least one of these points $P_\tau$ for any $k$ large, (2)
$P^{(k)}$ changes by a vanishingly small amount as $k$ becomes large,
and (3) the points $P_\tau$ are mutually isolated, it follows that
$P^{(k)}$ converges to one such point $P_\tau$ as $k\rightarrow \infty$.
One might expect that such a limit point must actually be a fixed point
$P_f$, and indeed this is true.

In summary, these arguments establish that we have global convergence to
a fixed point. We have already demonstrated with toy counterexamples
that this limit point may differ from the true density matrix $P$. The remainder of the proof consists in establishing that for generic initial guess $P^{(0)}$, the limit point is in fact equal to $P$, not some other fixed point $P_f$.

As mentioned above, the conditions of Section \ref{subsec:genericity} ensure
that $A+\wt{B}[P_f]$ has a positive spectral gap for any fixed point
$P_f$, which in turn ensures that $\mf{F}$ is smooth near each of the
fixed points $P_f$. In Section \ref{subsec:linearizationFP}, we then carry out a linearization-based analysis similar to that of
Section \ref{subsec:linearLocal}, which reveals that all pathological fixed
points $P_f \neq P$ are unstable. To complement this perspective, in Section \ref{subsec:fpsp} we exhibit a functional that is monotone nonincreasing along the iterates, for which the fixed points are critical points, among which the true density matrix $P$ is the only local minimum.

In summary, at this point we have established that we have global
convergence to a fixed point and moreover that all fixed points but the
true density matrix $P$ are unstable. This picture is already strongly
suggestive that for a generic initial guess, we will never converge to a
``bad'' fixed point. In the language of dynamical systems, each bad fixed point has a local stable manifold of strictly positive codimension, hence of measure zero. If we were to have convergence to a bad fixed point, it would mean that for all $k$ sufficiently large, $P^{(k)}$ lies on one such local stable manifold for a bad fixed point $P_f \neq P$. What we want to show, then, is that it is impossible for the iteration map $\mf{F}$ to collapse a set of positive measure to a set of zero measure, i.e., that for $S$ of measure zero, $\mf{F}^{-1}(S)$ is of measure zero as well. The proof of this result, Lemma \ref{lem:preimage}, is rather involved and is given in Appendix \ref{sec:eggOnBarn}. A key difficulty is that $\mf{F}$ is not a diffeomorphism (which would render the lemma immediate), nor even is it continuous. Once established, by the above reasoning Lemma \ref{lem:preimage} completes the proof of Theorem
\ref{thm:main3}.

\subsection{Eigenvalue monotonicity}\label{subsec:eigenMono}
We now highlight a significant feature of Algorithm \ref{alg:acelineig},
which is the key for the proof of global convergence properties.
\begin{lemma}[Eigenvalue monotonicity]\label{lem:eigenMono}
For $i=1,\ldots,n$,
\[
\lambda_i^{(k)} := \lambda_i  \big\{ A+\wt{B}[P^{(k-1)}]\big\}
\]
is non-increasing in $k$.
\end{lemma}
\begin{proof}
Let $v_i^{(k)}\in \mathrm{Im}(P^{(k)})$ be orthonormal eigenvectors of $ A+\wt{B}[P^{(k-1)}]$ corresponding to the eigenvalues $\lambda_i^{(k)}$ for $i=1,\ldots,n$, and let $S_i^{(k)} = \mathrm{span}\{v_1^{(k)} ,\ldots, v_i^{(k)} \}$. Then we compute, for $i=1,\ldots,n$,

\begin{eqnarray*}
\lambda_{i}^{(k+1)} & = & \min_{\dim S=i}\max_{u\in S\backslash\{0\}}\frac{u^{*}\big(A+\wt{B}[P^{(k)}]\big)u}{u^{*}u}\\
 & \leq & \max_{u\in S_i^{(k)} \backslash\{0\}}\frac{u^{*}\big(A+\wt{B}[P^{(k)}]\big)u}{u^{*}u}\\
 & \overset{\mathrm{(i)}}{=} & \max_{u\in S_i^{(k)} \backslash\{0\}}\frac{u^{*}\left(A+B\right)u}{u^{*}u}\\
 & \overset{\mathrm{(ii)}}{\leq} & \max_{u\in S_i^{(k)} \backslash\{0\}}\frac{u^{*}\big(A+\wt{B}[P^{(k-1)}]\big)u}{u^{*}u}\\
 & = & \lambda_{i}^{(k)},
\end{eqnarray*}
where (i) follows from the fact that $B \equiv \wt{B}[P^{(k)}]$ on $\mathrm{Im}(P^{(k)})\supset S_i^{(k)}$ and (ii) follows from Lemma \ref{lem:negdef}. This completes the proof.
\end{proof}

Thus we may think of Algorithm \ref{alg:acelineig} as performing a
descent on the bottom $n$ eigenvalues of $ A+\wt{B}[Q]$ as $Q=P^{(k)}$
is updated iteratively. In order to achieve global convergence, we would need that these eigenvalues are globally minimized at $Q=P$. Indeed, this is the case:
\begin{lemma}[Global eigenvalue minimality]\label{lem:globalEigenMin}
For $i=1,\ldots,n$ and all density matrices $Q$,
\[
\lambda_i =  \lambda_i  \big\{ A+\wt{B}[P]\big\} \leq \lambda_i  \big\{ A+\wt{B}[Q]\big\}.
\]
\end{lemma}
\begin{proof}
Let $v_i^{Q}\in \mathrm{Im}(\mf{F}(Q))$ be orthonormal eigenvectors of $ A+\wt{B}[Q]$ corresponding to the eigenvalues $\lambda_i  \big\{ A+\wt{B}[Q]\big\}$ for $i=1,\ldots,n$, and let $S_i^{Q} = \mathrm{span}\{v_1^{Q} ,\ldots, v_i^{Q} \}$. Again we compute, for $i=1,\ldots,n$,

\begin{eqnarray*}
\lambda_{i} & = & \min_{\dim S=i}\max_{u\in S\backslash\{0\}}\frac{u^{*}\left(A+B\right)u}{u^{*}u}\\
 & \leq & \max_{u\in S_i^{Q}\backslash\{0\}}\frac{u^{*}\left(A+B\right)u}{u^{*}u}\\
 & \leq & \max_{u\in S_i^{Q}\backslash\{0\}}\frac{u^{*}\big(A+\wt{B}[Q]\big)u}{u^{*}u}\\
 & = & \lambda_i  \big\{ A+\wt{B}[Q]\big\}.
\end{eqnarray*}
\end{proof}

We now examine some consequences of eigenvalue monotonicity with a view
toward establishing a global convergence result. First, from Lemma
\ref{lem:eigenMono} we have the immediate corollary.
\begin{corollary}[Eigenvalue convergence]\label{cor:eigenConv}
$\lim_{k\rightarrow\infty} \lambda_i ^{(k)}$ exists for $i=1,\ldots,n$.
\end{corollary}

\begin{remark}
The convergence rate of the eigenvalues $\lambda_i^{(k)}$ is not yet known at this stage in the proof. However, once global convergence of $P^{(k)}$ is established, it will follow that the asymptotic linear rate of convergence of the eigenvalues $\lambda_i^{(k)}$ is twice that of $P^{(k)}$ (which is in turn established in Theorem \ref{thm:main2}). This is the case because the true density matrix $P$ is a stationary point for the sum of the lowest $k$ eigenvalues of $A+\wt{B}[P]$ for any $k \leq n$.
\end{remark}

From this corollary and a refinement of earlier arguments, we derive the following result, which will be instrumental in establishing global convergence. The main idea of this result is that a small change in eigenvalues across one iteration is only possible if the density matrix changes by a correspondingly small amount.

\begin{lemma}\label{lem:densMatChange}
There exists a constant $C>0$ (depending only on $B,n$) such that
\[
\Vert P^{(k)} - P^{(k-1)} \Vert_2 \leq C \sqrt{\delta^{(k)}}
\]
for all $k$, where
\[
\delta^{(k)} := \sum_{i=1}^n \left( \lambda_i^{(k)}- \lambda_i^{(k+1)} \right).
\]
It follows (by Corollary \ref{cor:eigenConv}) that $\Vert P^{(k)} - P^{(k-1)} \Vert_2 \rightarrow 0$ as $k \rightarrow \infty$.
\end{lemma}
\begin{proof}
As in the proof of Lemma \ref{lem:eigenMono}, let $v_i^{(k)}\in \mathrm{Im}(P^{(k)})$ be orthonormal eigenvectors of $ A+\wt{B}[P^{(k-1)}]$ corresponding to the eigenvalues $\lambda_i^{(k)}$ for $i=1,\ldots,n$, and let $S_i^{(k)} = \mathrm{span}\{v_1^{(k)} ,\ldots, v_i^{(k)} \}$.

\begin{eqnarray*}
\sum_{i=1}^n \lambda_i^{(k+1)} & = \ & \inf_{u_1,\ldots,u_n\ \mathrm{orthonormal}} \left\{ \sum_{i=1}^n u_i^* \big(A+\wt{B}[P^{(k)}]\big)u_i \right\} \\
& \leq \ & \sum_{i=1}^n \big(v_i^{(k)}\big)^* \big(A+\wt{B}[P^{(k)}]\big) v_i^{(k)} \\ 
& =\  & \sum_{i=1}^n \big(v_i^{(k)}\big)^* \big(A+B\big) v_i^{(k)} \\
& =\  & \sum_{i=1}^n \big(v_i^{(k)}\big)^* \big(A+\wt{B}[P^{(k-1)}]\big) v_i^{(k)} - 
\sum_{i=1}^n \big(v_i^{(k)}\big)^* \big(\wt{B}[P^{(k-1)}] - B\big) v_i^{(k)} \\
& =\ &  \sum_{i=1}^n \lambda_i^{(k)} - 
\sum_{i=1}^N \big(v_i^{(k)}\big)^* P^{(k)} \big(\wt{B}[P^{(k-1)}] - B\big) v_i^{(k)} \\
& = \ & \sum_{i=1}^n \lambda_i^{(k)} -
\mathrm{Tr}\left[P^{(k)} \big(\wt{B}[P^{(k-1)}] - B \big) \right],
\end{eqnarray*}
where `$\mathrm{Tr}$' denotes the matrix trace.
We have
\begin{equation}\label{eqn:traceIneq}
\mathrm{Tr}\left[P^{(k)} \big(\wt{B}[P^{(k-1)}] - B \big) P^{(k)} \right]
\leq \sum_{i=1}^n \left(\lambda_i^{(k)} - \lambda_i^{(k+1)} \right) = \delta^{(k)}.
\end{equation}
At this point we should hope that the left-hand side of \eqref{eqn:traceIneq} provides an upper bound for some measure of the distance between $P^{(k)}$ and $P^{(k-1)}$, and indeed this will be the case.

We first prove the following lemma.
\begin{lemma}\label{lem:rotate}
There exists $t>0$ depending only on $B,n$ such that 
\begin{equation*}
\mathrm{Tr} \left[ R \big(\wt{B}[Q]-B \big) \right]
\geq t\cdot \mathrm{Tr} \left[ R \big(I - Q \big) \right]
\end{equation*}
for all density matrices $Q$ and $R$.
\end{lemma}
\begin{proof}
Note that 
\[
\lambda_{\min} \left\{ Q+ \big(\wt{B}[Q]-B \big) \right\} > 0,
\]
for all density matrices $Q$. By the continuity of $\wt{B}$ on density matrices and $\lambda_{\min}$ on Hermitian matrices, as well as the compactness of the space of density matrices, it follows that there exists $t>0$ such that 
\[
Q+ \big(\wt{B}[Q]-B \big) \succeq t
\]
for all density matrices $Q$. Furthermore, we can write
\[
\wt{B}[Q]-B = (I-Q) \left[ Q+ \big(\wt{B}[Q]-B \big) \right] (I-Q) \succeq t(I-Q).
\]
Now the trace of a product of positive semidefinite matrices is nonnegative, so 
\[
\mathrm{Tr} \left( R \big[(\wt{B}[Q]-B) - t(I-Q)  \big] \right) \geq 0,
\]
for all density matrices $Q,R$, which yields the lemma.
\end{proof}

We now resume the proof of Lemma \ref{lem:densMatChange}. Let $R = P^{(k)}$ and $Q = P^{(k-1)}$ in Lemma \ref{lem:rotate} to obtain
\begin{equation*}
\mathrm{Tr} \left[ P^{(k)} \big(\wt{B}[P^{(k-1)}]-B \big)\right]
\geq t\cdot \mathrm{Tr} \left[ P^{(k)} \big(I - P^{(k-1)} \big) \right],
\end{equation*}
and combine with \eqref{eqn:traceIneq} to give
\begin{equation*}
\mathrm{Tr} \left[ P^{(k)} \big(I - P^{(k-1)} \big) \right]
\leq \alpha \delta^{(k)},
\end{equation*}
where $\alpha := t^{-1} > 0$. Now
\begin{equation*}
\mathrm{Tr} \left[ P^{(k)} \big(I - P^{(k-1)} \big) \right] = \mathrm{Tr} \left[ P^{(k)} - P^{(k)} P^{(k-1)} \right] = n -  \mathrm{Tr} \left[ P^{(k)} P^{(k-1)} \right],
\end{equation*}
so in fact we have
\begin{equation}\label{eqn:traceIneq2}
\mathrm{Tr} \left[ P^{(k)} P^{(k-1)}  \right]
\geq n - \alpha \delta^{(k)}.
\end{equation}
To conclude the proof, observe
\begin{eqnarray*}
\norm{ P^{(k)}- P^{(k-1)}}_2 ^2 &\le& \norm{ P^{(k)}- P^{(k-1)}}^2_F \\
&=&  \mathrm{Tr} \left[ (P^{(k)} - P^{(k-1)})(P^{(k)} - P^{(k-1)}) \right] \\
&=& 2n - 2 \cdot \mathrm{Tr} \left[ P^{(k)} P^{(k-1)}\right] \\ 
&\le & 2  \alpha \delta^{(k)},
\end{eqnarray*}
where we have used \eqref{eqn:traceIneq2} in the last line.
\end{proof}

Note carefully that Lemma \ref{lem:densMatChange} \textit{does not}
imply that the sequence $P^{(k)}$ is convergent. In particular, we do
not yet see that $P^{(k)}$ is Cauchy; we are only able to bound the
change in density matrix over a single iteration.
However, Lemma \ref{lem:densMatChange} does establish that for $k$
large, the density matrix $P^{(k)}$ is almost fixed by $\mf{F}$. 
Note that any fixed point $P_f$ is a projector of the form $P_f =
\sum_{i=1}^n u_i u_i^*$, where the $u_i$'s are eigenvectors of $A+B$. This
motivates the following lemma, which implies that for $k$ large,
$P^{(k)}$ is close to some point of this form.

\begin{lemma}\label{lem:almostEigen}
There exists a constant $C>0$ depending only on $A,B,n$ such that if
$\Vert \mf{F}(Q) - Q \Vert_2 \leq \epsilon$ for any density matrix $Q$,
then $Q = \sum_{i=1}^n u_i u_i ^* + M$, where the $u_i$'s are orthonormal eigenvectors of $A+B$ and $\Vert M \Vert _2 \leq C\epsilon$. 
\end{lemma}
\begin{proof}
Write $\mf{F}(Q) = \sum_{i=1}^n w_i w_i ^*$, where $w_1,\ldots,w_n$ are
orthonormal eigenvectors of $A+\wt{B}[Q]$ with corresponding eigenvalues
$\mu_1 \leq \cdots \leq \mu _n$. Let $z_i = Q w_i$ for $i=1,\ldots,n$.
Observe that 
\[
\left[(A+B)+(\wt{B}[Q]-B)(\mf{F}(Q)-Q)\right] w_i = \mu_i w_i
\]
for $i=1,\ldots,n$, since $(\wt{B}[Q]-B)Q = 0$ and $\mf{F}(Q)w_i = w_i$. Therefore
\begin{eqnarray*}
\Vert (A+B)w_i - \mu_i w_i \Vert & =& \Vert (\wt{B}[Q]-B)(\mf{F}(Q)-Q)w_i\Vert \\
& \leq & \Vert \wt{B}[Q]-B\Vert_2 \Vert\mf{F}(Q)-Q \Vert_2.
\end{eqnarray*}
We assume $\Vert\mf{F}(Q)-Q \Vert_2 \leq \epsilon$ as in the statement of the theorem, and recall $\Vert \wt{B}[Q]-B\Vert_2 \leq \Vert B\Vert_2$, so we have shown that 
\begin{equation}\label{eqn:almostEigenvector}
\Vert (A+B)w_i - \mu_i w_i \Vert \leq C\epsilon,
\end{equation}
where $C = \Vert B\Vert_2$.
In other words, if $\epsilon$ is small, then $w_i$ nearly satisfies the
condition of being eigenvectors of $A+B$ with the corresponding eigenvalue $\mu_i$. We now aim to show that this implies that each $w_i$ is in fact close to some eigenvector of $A+B$. 
We remark that the discussion below is related to the ``sin $\theta$ theorem'' of Davis and Kahan
\cite{DavisKahan1970}, which characterizes the relation between 
the error of an approximate eigenvector and its residual.

To this end, let $v_1,\ldots,v_N$ be orthonormal eigenvectors of $A+B$ with corresponding eigenvalues $\lambda_1 \leq \cdots \leq \lambda_N$, and write $w_i = \sum_{j=1}^N c_{ij}v_j$. Then 
\[
\Vert (A+B)w_i - \mu_i w_i \Vert^2 = \left\Vert \sum_{j=1}^N c_{ij} (\lambda_j - \mu_i) v_j \right\Vert^2 = \sum_{j=1}^N \vert c_{ij}\vert^2 \vert \lambda_j -\mu_i \vert^2.
\]
Combining with \eqref{eqn:almostEigenvector} yields
\begin{equation}\label{eqn:productIneq}
\vert c_{ij}\vert^2 \vert \lambda_j -\mu_i \vert^2 \leq C^2 \epsilon^2.
\end{equation}

Let $\delta >0$ be smaller than the gap between any pair of \textit{distinct} eigenvalues of $A+B$. (Note carefully that this is still possible even if $A+B$ has repeated eigenvalues.) Fix $i$ for the moment, and decompose
\[
w_i = \sum_{\{j\,:\,\vert \lambda_j - \mu_i \vert > \delta \} } c_{ij} v_j
+ \underbrace{\sum_{\{j\,:\,\vert \lambda_j - \mu_i \vert \leq \delta \}} c_{ij} v_j}_{=: \tilde{u}_i}. \\
\]
Notice that if $\vert \lambda_j - \mu_i \vert > \delta$, then $\vert c_{ij}\vert^2 \leq C^2 \epsilon^2 / \delta^2$ by \eqref{eqn:productIneq}. Thus 
\[
\Vert w_i - \tilde{u}_i \Vert^2 = \sum_{\{j\,:\,\vert \lambda_j - \mu_i \vert > \delta \}} \vert c_{ij}\vert^2 \leq \frac{N C^2}{\delta^2} \epsilon^2.
\]
In particular, for $\epsilon$ sufficiently small, $\Vert w_i -\tilde{u}_i\Vert <1$, which implies that $\tilde{u}_i \neq 0$.

By the definition of $\delta$, there is at most one element in the
set $\{\lambda_j\,:\,\vert \lambda_j -\mu _i \vert \leq \delta \}$. But since $\tilde{u}_i \neq 0$, there must also be at least one element. We denote this element by $\lambda[i]$. Observe that $\tilde{u}_i$ is in the $\lambda[i]$-eigenspace of $A+B$.

We have established (for a possibly enlarged constant $C$ depending only
on $A,B$) that if $\epsilon$ is sufficiently small, then 
\[
\Vert w_i - \tilde{u}_i \Vert \leq C \epsilon.
\]
Then the $\tilde{u}_i$ must be linearly independent for $\epsilon$ sufficiently small. Moreoever, since the $w_i$ are orthonormal, this implies (possibly enlarging $C$ once again) that
\[
\left\Vert \mf{F}(Q) - \tilde{U} (\tilde{U}^* \tilde{U})^{-1} \tilde{U}^* \right\Vert_2 =
\left\Vert W(W^* W)^{-1} W^* - \tilde{U} (\tilde{U}^* \tilde{U})^{-1} \tilde{U}^* \right\Vert_2
\leq C \epsilon
\]
for $\epsilon$ sufficiently small, where
$\tilde{U}:=[\tilde{u}_1,\ldots,\tilde{u}_n]$, so $\tilde{U}
(\tilde{U}^* \tilde{U})^{-1} \tilde{U}^*$ is the orthogonal projector
onto the span of the $\tilde{u}_i$, and likewise $W:=[w_1,\ldots,w_n]$.
Now the $\tilde{u}_i$'s are unnormalized eigenvectors of $A+B$ with
possibly repeated eigenvalues, hence possibly not orthonormal or even
orthogonal. However, $\mathrm{span}\{\tilde{u}_1,\ldots,\tilde{u}_n\}$
is invariant under $A+B$, hence can also be endowed with an orthonormal
basis of eigenvectors $u_1,\ldots,u_n$ of $A+B$. This yields the equivalent orthogonal projector $\sum_{i=1}^n u_i u_i^*$. Now since $\Vert \mf{F}(Q)-Q\Vert_2 \leq \epsilon$, this means that (enlarging $C$ again)
\[
\left\Vert Q - \sum_{i=1}^n u_i u_i^* \right\Vert_2 \leq C \epsilon
\]
for $\epsilon$ sufficiently small.

This establishes the statement of the lemma under the condition that
$\epsilon$ is assumed sufficiently small. But since the space of density
matrices is compact, there exists $K>0$ such that $\Vert Q - \sum_{i=1}^n v_i v_i^* \Vert_2 \leq K$ for any density matrix $Q$. By enlarging $C$ sufficiently the lemma is proved.
\end{proof}

\subsection{Genericity assumptions}\label{subsec:genericity} We will
impose some assumptions that will ensure  that $\mf{F}$ has finitely
many fixed points and that at each fixed point $P_f$, $A+\wt{B}[P_f]$
has a spectral gap, so that $\mf{F}(P_f)$ can be defined canonically. We
will argue that these assumptions hold generically, i.e., can be made to
hold by an arbitrarily small perturbation of the eigenvalue problem \eqref{eqn:lineig}. Our genericity assumptions will allow us (1) to prove that $P^{(k)}$ converges to a fixed point and (2) to perform a first-order analysis of $P_f$ near each fixed point.

\begin{assumption}\label{assump:distinctEigen}
Assume that $A+B$ has distinct eigenvalues $\lambda _1 < \cdots <\lambda_N$ corresponding to orthonormal eigenvectors $v_1,\ldots,v_N$.
\end{assumption}

This can be guaranteed by replacing $A$ or $B$ with a suitable arbitrarily small random perturbation of $A$ or $B$ (see, e.g., Section 1.3 of \cite{TaoRMT}).

\begin{assumption}\label{assump:spectralGap}
For $\tau:\{1,\ldots,n\}\rightarrow \{1,\ldots,N\}$ increasing, let $P_\tau = \sum_{i=1}^n v_{\tau(i)} v_{\tau(i)}^*$, and let $S_\tau = \mathrm{Im}(P_\tau)$. Assume that for all such $\tau$,
\[
\lambda_{\tau(n)} \neq \lambda_{\min} \left\{ \big(A+\wt{B}[P_{\tau}]\big)\big\vert_{S_\tau ^\perp} \right\},
\]
or, equivalently,
\[
\lambda_{\max} \left\{ \big(A+\wt{B}[P_{\tau}]\big)\big\vert_{S_\tau } \right\} \neq \lambda_{\min} \left\{ \big(A+\wt{B}[P_{\tau}]\big)\big\vert_{S_\tau ^\perp} \right\}.
\]
\end{assumption}

We now provide some interpretation for Genericity Assumption \ref{assump:spectralGap}. If 
\[
\lambda_{\tau(n)} < \lambda_{\min} \left\{ \big(A+\wt{B}[P_{\tau}]\big)\big\vert_{S_\tau ^\perp} \right\},
\]
then $P_{\tau}$ is a fixed point of $\mf{F}$. Moreover,
$A+\wt{B}[P_{\tau}]$ has a positive spectral gap, so $\mf{F}(P_{\tau})$
is canonically defined. Meanwhile, if 
\[
\lambda_{\tau(n)} > \lambda_{\min} \left\{ \big(A+\wt{B}[P_{\tau}]\big)\big\vert_{S_\tau ^\perp} \right\},
\]
then $P_{\tau}$ is \textit{definitely not} a fixed point of $\mf{F}$ (though it does not necessarily follow that $A+\wt{B}[P_{\tau}]$ has a positive spectral gap). Lastly, if 
\[
\lambda_{\tau(n)} = \lambda_{\min} \left\{ \big(A+\wt{B}[P_{\tau}]\big)\big\vert_{S_\tau ^\perp} \right\},
\]
then $A+\wt{B}[P_{\tau}]$ has zero spectral gap, and $P_{\tau}$ \textit{may or may not} be a fixed point, depending on the arbitrary choice made for $\mf{F}(P_{\tau})$. This is precisely the scenario that Genericity Assumption \ref{assump:spectralGap} rules out.

We will argue that Genericity Assumption \ref{assump:spectralGap} can be guaranteed by replacing (if necessary) $B$ with $B-t$ for all but finitely many $t \geq 0$. Note that this does not change the eigenspaces of $A+B$ and only affects the eigenvalues by shifting them all downward by $t$.
We first provide a characterization of fixed points of $\mf{F}$.

\begin{lemma}[Characterization of fixed points]\label{lem:fixedPt}
Suppose that $P_f$ is a fixed point of $\mf{F}$. Then we can write
\begin{equation}\label{eqn:fixedPtDecomp}
A+\wt{B}[P_f] = \sum_{i=1}^N \mu_i z_i z_i ^*,
\end{equation}
where $z_1,\ldots,z_N$ are orthonormal eigenvectors of $A+\wt{B}[P_f]$
with corresponding eigenvalues $\mu_1 \leq \cdots \leq \mu_N$. Moreover
$z_1,\ldots,z_n$ are eigenvectors of $A+B$ forming an orthonormal basis
of $\mathrm{Im}(P_f)$. Consequently $\mu_i = \lambda_{\tau(i)}$, where
$\tau:\{1,\ldots,n\}\rightarrow\{1,\ldots,N\}$ is increasing, and $P_f =
P$ if and only if $\mu_i = \lambda_i$ for $i=1,\ldots,n$. Otherwise
$\mu_n \geq \lambda_n + \lambda_g$.
\end{lemma}
\begin{proof}
Let $P_f$ be a fixed point of $\mf{F}$. Referring to Definition \ref{def:IterOp}, we see that then $A+\wt{B}[P_f]$ maps $\mathrm{Im}(P_f)$ into itself. But $A+\wt{B}[Q] \equiv A+B$ on $\mathrm{Im}(P_f)$, so $A+B$ maps $\mathrm{Im}(P_f)$ into itself. $A+B$ can then be considered (via restriction) as a self-adjoint operator $\mathrm{Im}(P_f) \rightarrow \mathrm{Im}(P_f)$, so $\mathrm{Im}(P_f)$ has an orthonormal basis of eigenvectors $z_1,\ldots,z_n$ of $A+B$ with corresponding eigenvalues $\mu_1 \leq  \ldots \leq \mu_n$.

Since $A+\wt{B}[P_f]$ is self-adjoint, we also have that $A+\wt{B}[P_f]$ maps $\mathrm{Im}(P_f) ^\perp$ into itself, so $\mathrm{Im}(P_f) ^\perp$ has an orthonormal basis $z_{n+1},\ldots,z_N$ of eigenvectors of $A+\wt{B}[P_f]$ with corresponding eigenvalues $\mu_{n+1} \leq  \ldots \leq \mu_N$. The decomposition of \eqref{eqn:fixedPtDecomp} follows, provided we can show that $\mu_{n} \leq \mu_{n+1}$.

We will establish this now. First observe the general fact that for any density matrix $Q$, if $u\in \mathrm{Im}(\mf{F}(Q))$ is a unit vector and 
\[
u^* \big(A+\wt{B}[Q]\big) u > \hat{u}^* \big(A+\wt{B}[Q]\big) \hat{u}
\]
for some unit vector $\hat{u}$, then $\hat{u}\in \mathrm{Im}(\mf{F}(Q))$ as well. Now suppose for contradiction that $\mu_{n} > \mu_{n+1}$. Then considering $z_{n}$, $z_{n+1}$, and $P_f$ in the places of $u$, $\hat{u}$, and $Q$, respectively, we conclude that $z_{n+1} \in \mathrm{Im}(\mf{F}(P_f))$. But since $P_f$ is a fixed point of $\mf{F}$, this means that $z_{n+1} \in \mathrm{Im}(P_f)$, which is impossible since $0\neq z_{n+1} \in \mathrm{Im}(P_f)^\perp$.

Now if $\mu_i \neq \lambda_i$ for some $i\in\{1,\ldots,n\}$, we must have $\mu_n = \lambda_m$ for some $m>n$, so $\mu_n \geq \lambda_{n+1} = \lambda_n + \lambda_g$. In this case, we cannot have $P_f = P$, for if this were true then $\mathrm{Im}(P)$ would contain an eigenvector of $A+B$ with eigenvalue greater than $\lambda_n$.

Lastly, suppose that $\mu_i = \lambda_i$ for $i=1,\ldots,n$. Then $\left.(A+B)\right\vert_{\mathrm{Im}(P_f)} \preceq \lambda_n$. Since $A+B$ has a spectral gap, we must have that $\mathrm{Im}(P_f) = \mathrm{Im}(P)$, i.e., $P_f = P$.
\end{proof}

Recall that we would like to establish that Genericity Assumption \ref{assump:spectralGap} holds generically by replacing $B$ with $B-t$.

By Genericity Assumption \ref{assump:distinctEigen}, $A+B$ has only finitely many distinct eigenvectors (up to scaling). Then by Lemma \ref{lem:fixedPt}, $\mf{F}$ can only have finitely many fixed points. More precisely, this is the case because by Lemma \ref{lem:fixedPt} the candidates for fixed points are limited to projectors of the form $P_{\tau}=\sum_{i=1}^n v_{\tau(i)} v_{\tau(i)}^*$, where $\tau:\{1,\ldots,n\}\rightarrow\{1,\ldots,N\}$ is increasing.

For such $\tau$, note that $S_\tau = \mathrm{span}\,\{v_{\tau(1)},\ldots,v_{\tau(n)}\}$ is an invariant subspace for $A+\wt{B}[P_{\tau}]$, and hence so is $S_\tau ^\perp$. Let
\[
\mu_\tau := \lambda_{\min} \left\{ \big(A+\wt{B}[P_{\tau}]\big)\big\vert_{S_\tau ^\perp} \right\}.
\]
If $\mu_\tau < \lambda_{\tau(n)}$, then by Lemma \ref{lem:fixedPt}, $P_\tau$ is not a fixed point. If $\mu_\tau > \lambda_{\tau(n)}$, then evidently $P_\tau$ is a fixed point. If $\mu_\tau > \lambda_{\tau(n)}$, then $P_\tau$ may or may not be a fixed point, since the spectral gap of $A+\wt{B}[P_{\tau}]$ is zero and the choice of $\mf{F}(P_\tau)$ is not canonical. This last event is precisely what we would like to rule out.

More precisely, we would like to guarantee that for all of the (finitely many) increasing functions $\tau:\{1,\ldots,n\}\rightarrow\{1,\ldots,N\}$, we have that $\mu_\tau \neq \lambda_{\tau(n)}$.

Define $B_t := B - t$ for $t\geq 0$, and consider replacing $B$ with
$B_t$ in the eigenvalue problem \eqref{eqn:lineig}. Accordingly, define
$\lambda_i (t)$ and $\mu_\tau(t)$ now as functions of $t \geq 0$.
Evidently $\lambda_i(t) = \lambda_i - t$.

We would like to get a handle on $\mu_{\tau}(t)$. Extend $\tau$ to a permutation on all of $\{1,\ldots,N\}$ (so $v_{\tau(n+1)},\ldots,v_{\tau(N)}$ forms a basis for $S_\tau ^\perp$), and recall from \eqref{eqn:tildeBblock} that we
can write 
\begin{equation*}
\left[\wt{B_t}[P_\tau ]\right]_{V_\tau} = \left(\begin{array}{cc}
B_{11}-t & B_{12}\\
B_{12}^{*} & B_{12}^{*}(B_{11}-t)^{-1}B_{12}
\end{array}\right),
\end{equation*}
for suitable blocks $B_{ij}$, where $V_\tau := [v_{\tau(1)},\ldots,v_{\tau(N)}]$. Since $B_{11}$ is negative definite, we have that \[ B_{12}^{*}(B_{11}-t)^{-1}B_{12} \succeq B_{12}^{*} B_{11}^{-1}B_{12} \] for all $t\geq 0$. It follows that $\mu_{\tau}(t) \geq \mu_{\tau}$.

Thus for every $\tau$, $f_\tau(t) := \mu_\tau (t) - \lambda_{\tau(n)}(t)$ is a strictly increasing function on $t\geq 0$, so $f_\tau$ can have at most one zero. Since there are only finitely many $\tau$ of interest, there can only be finitely many points at which $\mu_\tau (t) = \lambda_{\tau(n)}(t)$ for some $\tau$. This means that by replacing $B$ with $B-t$ for any $t\geq 0$ outside of a finite set, Genericity Assumption \ref{assump:spectralGap} holds.

\begin{remark}\label{rmk:assumptionSummary}
In summary, Genericity Assumptions \ref{assump:distinctEigen} and \ref{assump:spectralGap} can be made to hold by perturbing $A+B$ to have distinct eigenvalues, then in turn replacing $B$ with $B-t$ for any $t\geq 0$ outside of a finite set (the latter step yielding an equivalent eigenproblem). We have shown in particular that these assumptions imply that $\mf{F}$ has only finitely many fixed points and that, for any fixed point $P_f$ of $\mf{F}$, $A+\wt{B}[P_f]$ has a spectral gap. We keep these assumptions for the remainder of Section \ref{sec:global}.
\end{remark}

In particular---recalling that $\mathbf{H}_N$ and $\mathbf{S}_N$ denote the sets of $N\times N$ Hermitian and $N\times N$ real-symmetric matrices, respectively---we have the following:
\begin{lemma}\label{lem:genAssump}
Fix any $A\in\mathbf{H}_N$. Then Genericity Assumptions \ref{assump:distinctEigen} and \ref{assump:spectralGap} hold both (1) for all $B\in \mathbf{H}_N$ outside of a set of zero measure with respect to the Lebesgue measure on $\mathbf{H}_N$ and (2) for all $B\in \mathbf{S}_N$ outside of a set of zero measure with respect to the Lebesgue measure on $\mathbf{S}_N$.
\end{lemma}
\begin{remark}
Note that statement (1) does not imply statement (2). It is desirable to have both of these statements at our disposal for the following reason. If we are solving an eigenvalue problem where $B$ is real-symmetric, we would like to be able to guarantee that a small random \textnormal{real-symmetric} perturbation of $B$ will satisfy the Genericity Assumptions. With only the first statement, we could only guarantee that this would work for a random \textnormal{Hermitian} perturbation, which would almost surely introduce imaginary parts to all the entries of $B$. This would not be desirable from a computational perspective.
\end{remark}
\begin{proof}
To see that the statements (1) and (2) hold for Genericity Assumption
\ref{assump:distinctEigen} alone, refer to Section 1.3 of \cite{TaoRMT}.
Now the set $\mathrm{Sc}:=\{t\cdot I_N \,: \, t\in\mathbb{R}\}$ of
scalar matrices is a one-dimensional subspace of the both of the real
vector spaces $\mathbf{H}_N$ and $\mathbf{S}_N$. We have already argued
in the preceding discussion that for any $X\in \mathbf{H}_N$ (hence also
for any $X \in \mathbf{S}_N$), Genericity Assumption
\ref{assump:spectralGap} holds for a.e. choice of $B$ in the
one-dimensional space $X+\mathrm{Sc}$ (with respect to the
one-dimensional Lebesgue measure). By Fubini's theorem (considering the product decompositions $\mathbf{H}_N = \mathrm{Sc} + \mathrm{Sc}^\perp$ and $\mathbf{S}_N = \mathrm{Sc} + \mathrm{Sc}^\perp$, where the orthogonal complements are taken within $\mathbf{H}_N$ and $\mathbf{S}_N$, respectively), Genericity Assumption \ref{assump:spectralGap} holds for a.e. choice of $B$ in $\mathbf{H}_N$ with respect to the Lebesgue measure on $\mathbf{H}_N$ and a.e. choice of $B$ in $\mathbf{S}_N$ with respect to the Lebesgue measure on $\mathbf{S}_N$.
\end{proof}

\begin{corollary}
Genericity Assumptions \ref{assump:distinctEigen} and \ref{assump:spectralGap} hold for almost every pair $(A,B)$ in $\mathbf{H}_N \times \mathbf{H}_N$ (with respect to the Lebesgue measure on $\mathbf{H}_N \times \mathbf{H}_N$) and for almost every pair $(A,B)$ in $\mathbf{S}_N \times \mathbf{S}_N$ (with respect to the Lebesgue measure on $\mathbf{S}_N \times \mathbf{S}_N$).
\end{corollary}
\begin{proof}
This follows from Lemma \ref{lem:genAssump} and Fubini's theorem.
\end{proof}

\subsection{Global convergence to a fixed point, local convergence revisited}
\label{subsec:globalConvergenceToFixed}
We are now ready to prove that the adaptive compression method converges globally to a fixed point (though we do not yet address whether the fixed point is the true density matrix $P$).

\begin{proposition}\label{prop:globalConvToFixed}
$P^{(k)} \rightarrow P_f$ as $k\rightarrow \infty$ for some fixed point $P_f$ of $\mf{F}$.
\end{proposition}
\begin{proof}
As above let $v_1,\ldots,v_N$ be an orthonormal basis of eigenvectors of $A+B$ with corresponding eigenvalues $\lambda_1 \leq \cdots \leq \lambda_N$. Let $\mathcal{T}$ be the set of all density  matrices $P_\tau := \sum_{i=1}^n v_{\tau(i)} v_{\tau(i)}^*$ where $\tau:\{1,\ldots,n\}\rightarrow\{1,\ldots,N\}$ is increasing. Then Lemma \ref{lem:densMatChange} and Lemma \ref{lem:almostEigen} together imply that $\mathrm{dist}\left( P^{(k)}, \mathcal{T} \right) \rightarrow 0$ as $k \rightarrow \infty$.
However, since (by Lemma \ref{lem:densMatChange}) $\Vert P^{(k)} -
P^{(k-1)} \Vert \rightarrow 0$, and since $\mathcal{T}$ consists of only
finitely many (hence mutually isolated) points, it must be the case that
$P^{(k)} \rightarrow P_\tau$ for some $\tau$.  Below we show that
$P_{\tau}$ must also be a fixed point of $\mf{F}$.

Observe that, for all $k$,
\begin{equation}\label{eqn:boringInequality}
\lambda_{\max} \left\{ \big(A+\wt{B}[P^{(k-1)}]\big)\big\vert_{\mathrm{Im}(P^{(k)})} \right\} \leq \lambda_{\min} \left\{ \big(A+\wt{B}[P^{(k-1)}]\big)\big\vert_{\mathrm{Im}(P^{(k)})^\perp} \right\}.
\end{equation}
We will rewrite this inequality in a way that makes it clear that we can take a limit as $k\rightarrow \infty$. To this end, let $C \gg \Vert A \Vert_2 + \Vert B \Vert_2$, noting that $\Vert A \Vert_2 + \Vert B \Vert_2$ provides an upper bound on $\Vert A + \wt{B}[Q] \Vert_2$ for all density matrices $Q$, hence also an upper bound on the absolute value of the eigenvalues of $A + \wt{B}[Q]$. Then \eqref{eqn:boringInequality} is the same as
\begin{eqnarray*}
& & \lambda_{\max} \left\{ P^{(k)}\big(A+\wt{B}[P^{(k-1)}]\big) P^{(k)} - C\cdot\big(I-P^{(k)}\big) \right\} \\ 
& & \qquad \qquad \leq \lambda_{\min} \left\{ \big(I-P^{(k)}\big)\big(A+\wt{B}[P^{(k-1)}]\big)\big(I-P^{(k)}\big) + C\cdot P^{(k)} \right\}.
\end{eqnarray*}
Then by continuity and the convergence $P^{(k)} \rightarrow P_\tau$ we have 
\begin{eqnarray*}
& & \lambda_{\max} \left\{ P_\tau\big(A+\wt{B}[P_\tau]\big) P_\tau - C\cdot\big(I-P_\tau\big) \right\} \\ 
& & \qquad \qquad \leq \lambda_{\min} \left\{ \big(I-P_\tau\big)\big(A+\wt{B}[P_\tau]\big)\big(I-P_\tau\big) + C\cdot P_\tau \right\},
\end{eqnarray*}
i.e.,
\[
\lambda_{\tau(n)} = \lambda_{\max} \left\{ \big(A+\wt{B}[P_\tau]\big)\big\vert_{\mathrm{Im}(P_\tau)} \right\} \leq \lambda_{\min} \left\{ \big(A+\wt{B}[P_\tau]\big)\big\vert_{\mathrm{Im}(P_\tau)^\perp} \right\}.
\]
We have successfully passed \eqref{eqn:boringInequality} to the limit as $k\rightarrow \infty$.

If $P_\tau$ is not a fixed point, then
Genericity Assumption \ref{assump:spectralGap} implies that 
\[
\lambda_{\tau(n)} > \lambda_{\min} \left\{ \big(A+\wt{B}[P_{\tau}]\big)\big\vert_{\mathrm{Im}(P_\tau) ^\perp} \right\},
\]
yielding a contradiction.
\end{proof}

Next we see how the preceding results imply local convergence. Though we
have already provided a more refined local convergence result (complete
with a linear rate of convergence), it is noteworthy that local
convergence can be proved ``non-perturbatively''. 
For this proof, we will \textit{not} consider a linearization of $\mf{F}$ about $P$, instead relying only on eigenvalue monotonicity as the tool.

\begin{proposition}[Local convergence via eigenvalue monotonicity]\label{prop:localConvNonPerturb}
If $P^{(0)}$ is sufficiently close (in the sense of any given norm on $\mathbb{C}^{N\times N}$) to the true density matrix $P$, then $P = \lim_{k\rightarrow\infty} P^{(k)}$.
\end{proposition}

\begin{proof}
By the continuity of $\lambda_i \{\,\cdot\,\}$ and $\wt{B}$, for all choices of $P^{(0)}$ sufficiently close to $P$ we have that
\[
\lambda_i \big\{ A+\wt{B}[P^{(0)}] \big\} \leq \lambda_i + \lambda_g /2,
\]
for all $i=1,\ldots,n$. Then for such $P^{(0)}$, eigenvalue monotonicity (Lemma \ref{lem:eigenMono}) implies that
\[
\lambda_i \big\{ A+\wt{B}[P^{(k)}] \big\} \leq \lambda_i + \lambda_g /2,
\]
for all $k$, $i=1,\ldots,n$.

At a fixed point $P_f$, the bottom $n$ eigenvalues of $A+\wt{B}[P_f]$ are eigenvalues of $A+B$. If $P_f \neq P$, then one of these eigenvalues must be at least as large as $\lambda_{n+1} = \lambda_n + \lambda_g$.

But Proposition \ref{prop:globalConvToFixed} says that $P^{(k)} \rightarrow P_f$ for some fixed point $P_f$. By continuity, this is impossible if $P^{(0)}$ is taken as above, unless $P_f = P$.
\end{proof}

\subsection{Linearization around fixed points}
\label{subsec:linearizationFP}
In this section we repeat the first-order analysis of Section \ref{subsec:linearLocal} about an arbitrary fixed point of $\mf{F}$. We will see that all fixed points except the true density matrix are repulsive in a certain sense.

Let $P_f$ be a fixed point of $\mf{F}$. Then we can write $A+B = \sum_{i=1}^N \mu_i u_i u_i^*$, where $u_i$ are orthonormal eigenvectors of $A+B$ with eigenvalues $\mu_i$, and $P_f = \sum_{i=1}^n u_i u_i^*$. The picture is almost exactly the same as in Section \ref{subsec:linearLocal}, with the important difference that $\{\mu_1, \ldots, \mu_n \}$ might \textit{not} be the same as $\{\lambda_1,\ldots,\lambda_n\}$. Though $\mu_1,\ldots,\mu_n$ may not be the bottom eigenvalues of $A+B$, they \textit{are} the bottom eigenvalues of $A+\wt{B}[P_f]$, and in fact our genericity assumptions have guaranteed that $A+\wt{B}[P_f]$ enjoys a spectral gap. Thus $\mf{F}$ is smooth near $P_f$, and the same reasoning that yielded \ref{lem:composedJacobian} also yields the following.

\begin{lemma}\label{lem:composedJacobianGeneral}
 With notation as in the preceding discussion, for $\epsilon>0$ sufficiently small, let $P_f (\epsilon)$ be density matrices with $\Delta P = \lim_{\epsilon\to
  0} (P_f (\epsilon)- P_f)/\epsilon$. Then
  \begin{eqnarray*}
    D\mf{F}_{P_f}[\Delta P] &:=& \lim_{\epsilon\to 0} \frac{\mf{F}(P_f (\epsilon))-\mf{F}(P_f)}{\epsilon} \\
    &=& 
    \sum_{i=1}^n
    \left(P_f^\perp + \left(\wt{B}[P_f]-B\right)^\dagger \left(A+B- \mu_i\right) P_f^\perp \right)^\dagger    
    (\Delta P) u_i u_i^* + \mathrm{h.c.}
  \end{eqnarray*}
\end{lemma}

\begin{remark}\label{remark:ZmatrixGeneral}
  Write $B$ in the $U_N:= [u_1,\ldots,u_N]$ basis as 
\[
\left[ B \right]_{U_N} =
\left(\begin{array}{cc}
B_{11}^{P_f} & B_{12}^{P_f}\\
\left(B_{11}^{P_f}\right)^* & B_{22}^{P_f}
\end{array}\right).
\]
Then in this basis, the matrix of the linear transformation
\[
Z_i^{P_f} :=  \left(P_f^\perp + \left(\wt{B}[P_f]-B\right)^\dagger \left(A+B- \mu_i\right) P_f^\perp \right)^\dagger
\]
appearing in Lemma \ref{lem:composedJacobianGeneral} is given by
\begin{equation*}
  \left[ Z_i^{P_f} \right]_{U_{N}} =
\left(\begin{array}{cc}
0 & 0\\
0 & J_i^{P_f}
\end{array}\right),
\end{equation*}
where
\begin{equation}\label{eqn:JmatrixInBasis}
J_i^{P_f} := \left[ I_{N-n} + \left( -S_{22}^{P_f} \right)^{-1} (M_2 - \mu_i) \right]^{-1}
\end{equation}
and $S_{22}^{P_f}$ is shorthand for the Schur complement and is negative
definite,
and $M_2 := \mathrm{diag}(\mu_{n+1},\ldots,\mu_N)$. Note that if $P_f \neq P$, then for some $i\in \{1,\ldots,n\}$, $M_2 - \mu_i$ is diagonal with a strictly negative entry.
\end{remark}

We have, in contrast with Lemma \ref{lem:eigenJ}:

\begin{lemma}\label{lem:eigenJGeneral}
For $i=1,\ldots,n$, $J_i^{P_f}$ is diagonalizable with $\sigma\big(J_i^{P_f}\big) \subset (0,\infty)\backslash{\{1\}}$. Moreover, if $P_f \neq P$, then $\lambda_{\max}\big(J_i^{P_f}\big) > 1$ for some $i\in \{1,\ldots,n\}$.
\end{lemma}
\begin{proof}
In the proof we adopt notation from Remark \ref{remark:ZmatrixGeneral}.
Note that $\big( -S_{22}^{P_f} \big) = (\wt{B}[P_f]-B)_{22}$ and we can alternatively write
\begin{equation}\label{eqn:Jalternate}
Z_i^{P_f} = \left[ \left(A+\wt{B}[P_f]\right)_{22} - \mu_i) \right]^{-1} (\wt{B}[P_f]-B)_{22},
\end{equation}
where $(\,\cdot\,)_{22}$ denotes the lower-right block in the $u_i$ basis. As the product of two positive definite matrices, $Z_i^{P_f}$ has positive eigenvalues. This can be verified by conjugating by $(\wt{B}[P_f]-B)_{22}^{1/2}$.

By \eqref{eqn:JmatrixInBasis}, we have the set equality
\begin{eqnarray}
\sigma\left(Z_i^{P_f}\right) &=& \frac{1}{1+\sigma\Big[ (\wt{B}[P_f]-B)_{22}^{-1} (M_2 - \mu_i) \Big]} \nonumber \\
& = & \frac{1}{1+\sigma\Big[ (\wt{B}[P_f]-B)_{22}^{-1/2} (M_2 - \mu_i) (\wt{B}[P_f]-B)_{22}^{-1/2} \Big]}. \label{eqn:fixedPointZSpectrum}
\end{eqnarray}
Now the signs of the eigenvalues of 
\[
 (\wt{B}[P_f]-B)_{22}^{-1/2} (M_2 - \mu_i) (\wt{B}[P_f]-B)_{22}^{-1/2}
\]
are the same as those of $M_2 - \mu_i$. Since we have assumed (Genericity Assumption \ref{assump:distinctEigen}) that the eigenvalues of $A+B$ are distinct, $M_2 -\mu_i$ is diagonal with nonzero eigenvalues. Thus $1 \notin \sigma \big(Z_i^{P_f}\big)$. This establishes the first statement of the lemma.

Now assume that $P_f \neq P$, and choose $i\in \{1,\ldots,n\}$ such that $M_2 - \mu_i$ has a strictly negative entry. By \eqref{eqn:fixedPointZSpectrum}, $Z_i^{P_f}$ must then have an eigenvalue that is either negative or larger than $1$, but we have already established that $\sigma (Z_i^{P_f}) \subset (0,\infty)$, so the latter possibility must be true.
\end{proof}

Let us identify the tangent space $T_{P_f} \mathcal{D}$ of the space of density matrices at $P_f$ with $\mathbb{C}^{(N-n)n}$ as we did in Section \ref{subsec:linearLocal} at the fixed point $P$. Furthermore, let us identify $\mf{F}$ with a map $\mf{H}^{P_f}$ into $\mathbb{C}^{(N-n)n}$ defined on a neighborhood $\mathcal{U}$ of the origin in $\mathbb{C}^{(N-n)n}$ via a local diffeomorphism $\Phi :\mathcal{U} \rightarrow \mathcal{D}$, defined as in \eqref{eqn:tangentDiffeo} but with the $u_i$ now in the places of the $v_i$. Then like before we have that
\[
D\mf{H}^{P_f}_0 = 
\left(\begin{array}{cccc}
J_1^{P_f} & 0 & \cdots  &0 \\
0 & J_2^{P_f} & \cdots & 0 \\
\vdots & \vdots & \ddots &\vdots \\
0  &0 &\cdots  & J_n^{P_f}
\end{array}\right).
\]
By Lemma \ref{lem:eigenJGeneral}, $D\mf{H}^{P_f}_0$ is invertible and has no eigenvalues of modulus 1. Thus in the language of dynamical systems, $0$ is a hyperbolic fixed point of the dynamical system defined locally by $\mf{H}$ near the origin of $\mathbb{C}^{(N-n)n}$.

Assume that $P_f \neq P$, then $D\mf{H}^{P_f}_0$ has at least one
eigenvalue larger than $1$. If we identify $\mathbb{C}^{(N-n)n}$ with
$\mathbb{R}^{2(N-n)n}$, then the corresponding realification of
$D\mf{H}^{P_f}_0$ (an operator $\mathbb{R}^{2(N-n)n}\rightarrow
\mathbb{R}^{2(N-n)n}$) has the same eigenvalues (though now two copies
of each), hence two eigenvalues larger than $1$. By the stable manifold
theorem, the local stable manifold near the origin has real codimension
at least $2$, and there exists a neighborhood $\mathcal{V}\subset
\mathcal{U} $ of the origin such that if $X \in \mathcal{V}$ is not in
the local stable manifold, then $\left(\mf{H}^{P_f}\right)^m (X) \notin
\mathcal{V}$ for some $m \geq 1$ (refer to Theorems 10.6 and 10.7 of
\cite{TeschlODE}). In particular, this implies the following.
\begin{proposition}\label{prop:repulsive}
For any fixed point $P_f \neq P$, there is a neighborhood
$\mathcal{P}_f$ of $P_f$ in the space $\mathcal{D}$ of density matrices
and a subset $\widetilde{\mathcal{P}}_f \subset \mathcal{P}_f$ such that
$\mathcal{P}_f \backslash \widetilde{\mathcal{P}}_f$ has measure zero in
$\mathcal{D}$. If $P^{(k)} \in \widetilde{\mathcal{P}}_f$ for some $k$, then $P^{(m)} \notin \mathcal{P}_f$ for some $m>k$.
\end{proposition}
\begin{remark}
Note that a notion of measure zero can be defined on any smooth manifold
without actually choosing a measure or a Riemannian
structure. One way to produce such a notion is to pick any Riemannian metric on $\mathcal{D}$ and consider the corresponding volume measure. The measure zero sets with respect to this volume measure will be the same regardless of the choice of metric.
\end{remark}

Proposition \ref{prop:repulsive} is roughly saying that generically near a fixed point $P_f \neq P$, points are repelled from $P_f$. However, in the important special case that we are considering $A,B$ to be real-symmetric, note that $\mf{F}$ can be interpreted as a map from $\mathcal{D}_{\mathbb{R}}$ into itself. If we initialize with a real-symmetric guess, then we never leave the submanifold of real-symmetric projectors. The notion of full measure does not project from $\mathcal{D}$ to the submanifold $\mathcal{D}_{\mathbb{R}}$ of \textit{real-symmetric} orthogonal projectors, so Proposition \ref{prop:repulsive} \textit{does not} imply that generically (within $\mathcal{D}_{\mathbb{R}}$) points are repelled from $P_f$, and we must state this result separately.

\begin{proposition}\label{prop:repulsiveReal}
Suppose that $A$ and $B$ are real-symmetric, so $\mf{F}$ maps $\mathcal{D}_{\mathbb{R}}$ into itself. For any fixed point $P_f \neq P$ in $\mathcal{D}_{\mathbb{R}}$, there is a neighborhood $\mathcal{P}_f$ of $P_f$ in the space $\mathcal{D}_{\mathbb{R}}$ a subset $\widetilde{\mathcal{P}}_f \subset \mathcal{P}_f$ such that $\mathcal{P}_f \backslash \widetilde{\mathcal{P}}_f$ has measure zero in $\mathcal{D}_{\mathbb{R}}$. If $P^{(k)} \in \widetilde{\mathcal{P}}_f$ for some $k$, then $P^{(m)} \notin \mathcal{P}_f$ for some $m>k$.
\end{proposition}

\begin{proof}
By exactly the same proof is above (with $\mathbb{R}$ in place of $\mathbb{C}$), the restriction of the dynamical system defined by $\mf{F}$ to the submanifold $\mathcal{D}_{\mathbb{R}}$ has a hyperbolic fixed point at $P_f$ with invertible Jacobian and the same eigenvalues as before (though only one copy of each now, instead of two). In particular, the stable manifold within $\mathcal{D}_{\mathbb{R}}$ has real codimension at least one, and by the same reasoning as before, this implies the statement.
\end{proof}

Propositions \ref{prop:repulsive} and \ref{prop:repulsiveReal} formalize
the notion that any fixed point $P_f \neq P$ is repulsive. More
quantitatively speaking, based on Lemma \ref{lem:eigenJGeneral} we
expect to see ``linear divergence'' from any fixed point $P_f \neq P$
with rate $\max_{i=1,\ldots,n} \lambda_{\max}\big(Z_i^{P_f}\big) > 1$,
but we do not formalize this notion.

\subsection{Fixed points are saddle points}
\label{subsec:fpsp}
Consider the functional 
\begin{equation}
F(Q) = \mathrm{Tr}\big[\mf{F}(Q)\left(A+\wt{B}[Q]\right)\mf{F}(Q)\big] = \sum_{i=1}^n \lambda_i \left\{A+\wt{B}[Q]\right\}.
  \label{eqn:FQ}
\end{equation}
By eigenvalue monotonicity, $F(P^{(k)})$ is non-increasing in $k$. We
claim that fixed points of $\mf{F}$ are critical points of $F$ and that
a fixed point $P_f \neq P$ is not a local minimum. In fact, a fixed point
$P_f \neq P$ is a strict saddle point of $F$ in that it is a strict
local maximum of $F$ along some direction.

We state a more detailed version this fact formally in Proposition \ref{prop:saddlePoint} below. This result informs our understanding of the behavior of the iteration
near fixed points (see the discussion at the beginning of Section
\ref{subsec:globalConvergence}). However, we will
not use it directly to establish global convergence, and its proof is
largely computational, so we relegate this proof to an appendix.

\begin{proposition}\label{prop:saddlePoint}
Let $Q=Q(t)$ be a twice-differentiable density matrix-valued function of a single variable with $Q(0) = P_f$ a fixed point. Then $(F(Q))'(0) = 0$. Moreover, if $P_f \neq P$, then there exists such a function $Q(t)$ which additionally satisfies $(F(Q))''(0) < 0$. In fact, if we take $Q(t) = \Phi(t X)$, where $X\in \mathbb{C}^{(N-n)n}$ is an eigenvector of $D\mf{H}^{P_f}_0$ with eigenvalue larger than \textnormal{[}resp., smaller than\textnormal{]} $1$, then $(F(Q))''(0) < 0$ \textnormal{[}resp., $>0$\textnormal{]}.
\end{proposition}
\begin{proof}
See Appendix \ref{appendix:saddlePoint}.
\end{proof}

\subsection{Global convergence}\label{subsec:globalConvergence}

We already have a fairly complete picture of the global behavior of
Algorithm \ref{alg:acelineig}. In summary, we know that the fixed point
iteration converges to a fixed point, and we know that fixed points $P_f
\neq P$ are repulsive in the sense of Proposition \ref{prop:repulsive}.
We also know that such fixed points are strict saddle points of the
functional $F$ in Eq.~\eqref{eqn:FQ}. Moreover, along the repulsive
directions
at $P_f$ this functional has a strict local maximum. With a bit more
work, it is possible to show that for almost all $Q$ in a sufficiently
small neighborhood of $P_f$, there exists $m=m(Q)$ such that $F(\mf{F}^m
(Q)) < F(P_f)$. (We already know that generically such $Q$ must escape
the neighborhood, but when they do so, they should align with the
repulsive directions, so the value of $F$ must fall below $F(P_f)$.
We omit a formal proof of this fact.) Thus by eigenvalue monotonicity,
if we have converged sufficiently close to a fixed point $P_f \neq P$,
and if we apply a small random perturbation and then restart Algorithm
\ref{alg:acelineig} from this point, then we will converge to another
fixed point $\widetilde{P}_f$ with $F(\widetilde{P}_f) < F(P_f)$. Repeating this process
finitely many times will bring us to the desired fixed point $P$. This
suggests a satisfactory notion of the global convergence up to perturbation.

Nonetheless, it is still desirable to show that for almost every choice
of initialization $P^{(0)}\in\mathcal{D}$, Algorithm \ref{alg:acelineig}
converges to $P$. (Similarly in the special case of real-symmetric $A$
and $B$,  Algorithm \ref{alg:acelineig} converges to $P$ for a.e. choice
$P^{(0)}\in\mathcal{D}_{\mathbb{R}}$.) To use an analogy, a fixed point
$P_f \neq P$ is like an egg resting on top of a barn. We know that if we
apply a slight random perturbation to the egg, it will fall off the barn
and never return to the top. But we would like to show that it is
impossible for the egg to get stuck on top of the barn in the first
place!

This is indeed true, and the key lemma is the following.

\begin{lemma}[Egg on barn lemma]\label{lem:preimage}
If we fix $B \in \mathbf{H}_N$, then for almost any $A \in \mathbf{H}_N$ (with respect to the Lebesgue measure on $\mathbf{H}_N$), we have the following: if $S$ has zero measure in the space $\mathcal{D}$ of density matrices, then $\mf{F}^{-1}(S)$ also has zero measure in $\mathcal{D}$, where $\mf{F}$ is considered as a map $\mathcal{D}\rightarrow\mathcal{D}$.

Similarly, if we fix any $B \in \mathbf{S}_N$, then for almost any $A \in \mathbf{S}_N$ (with respect to the Lebesgue measure on $\mathbf{S}_N$), we have the following: 
if $S$ has zero measure in the space $\mathcal{D}_{\mathbb{R}}$ of density matrices, then $\mf{F}^{-1}(S)$ also has zero measure in $\mathcal{D}_{\mathbb{R}}$, where $\mf{F}$ is considered as a map $\mathcal{D}_{\mathbb{R}}\rightarrow\mathcal{D}_{\mathbb{R}}$.
\end{lemma}

The proof of Lemma \ref{lem:preimage} is technical. The main difficulty is
that $\mf{F}$ is not a diffeomorphism,
and indeed is not even continuous. However, it is real-analytic on an
open, connected subset of full measure, and this characterization allows us to rule out pathological behavior. We postpone the proof of Lemma \ref{lem:preimage} to
Appendix \ref{sec:eggOnBarn}. Let us now use this lemma to prove the global
convergence property.

Fix $K\in\{\mathbb{C},\mathbb{R}\}$, and assume that $B$ is such that
Lemma \ref{lem:preimage} applies. For a fixed point $P_f$, let
$S_{P_f}=\{Q\in \mathcal{D}_K \,: \, \mf{F}^k(Q) \rightarrow P_f\}$, so
$\mathcal{D}_K = \bigcup_{P_f\,\mathrm{fixed}} S_{P_f}$. If $P_f \neq
P$, then by Proposition \ref{prop:repulsive} or Proposition
\ref{prop:repulsiveReal}, for any $Q \in S_{P_f}$, we must have that
$\mf{F}^k(Q) \in \mathcal{P}_f \backslash \widetilde{\mathcal{P}}_f$ for
some $k = k(Q) \geq 0$. This implies that $S_{P_f} \subset
\bigcup_{k\geq 0} \mf{F}^{-k}(\mathcal{P}_f \backslash
\widetilde{\mathcal{P}}_f)$. But by Proposition \ref{prop:repulsive},
$\mathcal{P}_f \backslash \widetilde{\mathcal{P}}_f$ has measure zero.
By Lemma \ref{lem:preimage} (and induction),
$\mf{F}^{-k}(\mathcal{P}_f \backslash \widetilde{\mathcal{P}}_f)$ has
measure zero for all $k\geq 0$. Consequently, $S_{P_f}$ has measure zero
for all $P_f \neq P$. Hence $S_P$ has full measure, as desired. This completes the proof of
Theorem~\ref{thm:main3}.

\section*{Acknowledgment} 
The work of L. L. is partially supported by the National Science
Foundation under grant DMS-1652330, the Alfred P. Sloan fellowship, and
the DOE Center for Applied Mathematics for Energy Research Applications
(CAMERA) program. The work of M.L. is partially supported by the National Science Foundation Graduate Research Fellowship Program under grant DGE-1106400.
\appendix

\section{Derivative calculations in linearization}\label{appendix:calculations}

\proofof{Lemma \ref{lem:contour}}  The assumption of the positive spectral gap guarantees the existence
  of a simple contour $\mc{C}$ in the complex plane surrounding only
  the lowest $n$ eigenvalues of $H$. Using the contour integral representation of the density matrix $P$, we
  have
  \begin{equation}
    P = \frac{1}{2\pi\I} \oint_{\mc{C}} (z-H)^{-1} \ud z.
    \label{eqn:Pcontour}
  \end{equation}
  Assume that $\epsilon$ is small enough so that the contour $\mc{C}$ only
  surrounds the lowest $n$ eigenvalues of $H+\epsilon \Delta H$ as well.
  Then
  \[
  P_\epsilon = \frac{1}{2\pi\I} \oint_{\mc{C}} (z-H-\epsilon \Delta H)^{-1} \ud z.
  \]
  Since
  \[
  (z-H-\epsilon \Delta H)^{-1} = (z-H)^{-1} + \epsilon (z-H)^{-1} \Delta
  H (z-H)^{-1} + O(\epsilon^2),
  \]
  we have
  \begin{equation}
    \lim_{\epsilon\to 0} \frac{P_\epsilon-P}{\epsilon} = \frac{1}{2\pi\I} \oint_{\mc{C}}
    (z-H)^{-1}\Delta H(z-H)^{-1} \ud z.
    \label{eqn:contour_2}
  \end{equation}

  Next, apply the spectral decomposition of $H$ to \eqref{eqn:contour_2} to obtain
  \begin{equation}
    \begin{split}
    D P_H[\Delta H] =&  \frac{1}{2\pi\I} \oint_{\mc{C}}
    (z-H)^{-1}\Delta H(z-H)^{-1} \ud z\\
    =& \sum_{i,a=1}^{N} \frac{1}{2\pi\I} \oint_{\mc{C}}
    (z-\mu_{a})^{-1} u_{a} (u_{a}^{*} \Delta H u_{i}) u_{i}^{*} 
    (z-\mu_{i})^{-1} \ud z\\
    =& 
    \sum_{i,a=1}^{N} \frac{1}{\mu_{i}-\mu_{a}}\left[\frac{1}{2\pi\I} \oint_{\mc{C}}
    \left(\frac{1}{z-\mu_{i}} - \frac{1}{z-\mu_{a}}\right) \ud
    z\right] u_{a} (u_{a}^{*} \Delta H u_{i}) u_{i}^{*} \\
    =&\sum_{i=1}^{n}\sum_{a=n+1}^{N} \frac{1}{\mu_{i}-\mu_{a}}
    u_{a} (u_{a}^{*} \Delta H u_{i}) u_{i}^{*} + \mathrm{h.c.}
    \label{eqn:contourB_2}
    \end{split}
  \end{equation}
  In the last equation of \eqref{eqn:contourB_2}, we have used the
  Cauchy integral formula. This establishes the first desired equality. For the second equality, simply collapse the inner sum over $a$. $\square$

\proofof{Lemma \ref{lem:fB}}
Recall that $\wt{B}[P_\epsilon] = B(P_\epsilon B P_\epsilon)^\dagger B =
BG(\epsilon) B$, where $G (\epsilon) := (P_\epsilon B
P_\epsilon)^\dagger$. We want to evaluate the derivative in $\epsilon$
of $G(\epsilon)$ at $\epsilon = 0$. To do so, we treat the pseudoinverse
as follows. Note that we can alternatively write
\[
G(\epsilon) = \left[ P_\epsilon B P_\epsilon + \lambda (I-P_\epsilon) \right]^{-1} - \lambda^{-1} (I-P_\epsilon)
\]
for any $\lambda >0$.
Then 
\begin{eqnarray}
G'(0)
& = & -\left[ P B P + \lambda (I-P) \right]^{-1}
\left[ (\Delta P) B P + P B (\Delta P) - \lambda \Delta P  \right]
\left[ P B P + \lambda (I-P) \right]^{-1} \nonumber \\
& & \qquad + \lambda^{-1} (\Delta P) \nonumber \\
& = &  -\left[ (P B P)^\dagger + \lambda^{-1} (I-P) \right]
\left[ P B (\Delta P) + \mathrm{h.c.}   \right]
\left[ (P B P)^\dagger + \lambda^{-1} (I-P) \right] \nonumber \\
& & \qquad + \lambda \left[ (P B P)^\dagger + \lambda^{-1} (I-P) \right]
\left[ \Delta P  \right]
\left[ (P B P)^\dagger + \lambda^{-1} (I-P) \right] \label{eqn:GderivTerm} \\
& & \qquad + \lambda^{-1} (\Delta P) \nonumber.
\end{eqnarray}

Note that $(P_\epsilon)^2 = P_\epsilon$, and evaluating the derivative of this equality at $\epsilon = 0$ yields $P (\Delta P)+ (\Delta P) P = (\Delta P)$. Then left- and right-multiplying both sides of this equality by $P$ yields $2 P (\Delta P) P = P (\Delta P) P$, so 
\[
P (\Delta P) P = 0.
\]
Observe that $(P B P)^\dagger =P(P B P)^\dagger P$, so there is significant cancellation in the second term of \eqref{eqn:GderivTerm}, which becomes
\[
(\Delta P) (P B P)^\dagger + \mathrm{h.c.}
\]
Then substituting into \eqref{eqn:GderivTerm}, we obtain
\begin{eqnarray*}
G'(0) & = & -\left[ (P B P)^\dagger + \lambda^{-1} (I-P) \right]
\left[ P B (\Delta P) + \mathrm{h.c.}   \right]
\left[ (P B P)^\dagger + \lambda^{-1} (I-P) \right] \nonumber \\
& & \qquad + [(\Delta P)(P B P)^\dagger  + \mathrm{h.c.}] + \lambda^{-1} (\Delta P) \nonumber.
\end{eqnarray*}
Now the preceding equality holds for \textit{any} $\lambda >0$. Thus taking the limit as $\lambda \rightarrow \infty$ establishes

\begin{eqnarray*}
G'(0) &= & \left[ -(P B P)^\dagger P B (\Delta P) (P B P)^\dagger
+(\Delta P) (P B P)^\dagger \right] + \mathrm{h.c.} \\
& = & \left(I - (P B P)^\dagger B \right) (\Delta P) (P B P)^\dagger + \mathrm{h.c.} 
\end{eqnarray*}
Then 
\[
D \wt{B}_P[\Delta P] = B G'(0) B = \left(B - \wt{B}[P] \right) (\Delta P) (P B P)^\dagger B + \mathrm{h.c.},
\]
as was to be shown. $\square$

\section{Proof of Lemma \ref{lem:invariantSubspace}}\label{appendix:sub-projector}

\proofof{Lemma \ref{lem:invariantSubspace}}
To ease the notation, let $B_\delta = B_\delta(0)$ and $T = DF(0) \in
\mathbb{R}^{k\times k}$. Let $Q_i$ be the orthogonal projector onto
$E_i$, and define $\rho(x) = \Vert Q_1 x \Vert_2$. Throughout the proof we
will use the shorthand notation $x=(x_1,x_2)$, e.g. $Q_1 x = (x_1,0)$ and
$\rho(x) = \Vert x_1 \Vert_2$.

Fix $\epsilon>0$ small enough such that $0 \prec T + \epsilon \prec 1$ and $0 \prec T\vert_{E_1} + \epsilon \prec \alpha$. Choose $\delta' \in (0,\delta)$ small enough so that $\Vert DF - T \Vert_2 \leq \epsilon$ on $B_{\delta'}$. Then $\Vert DF \Vert _2 \leq 1$ on $B_{\delta'}$, from which it follows that $F$ is non-expansive on $B_{\delta'}$, hence maps $B_{\delta'}$ into itself.

Then for the proof, we want to show that $\rho(F^k (x)) \leq \alpha^k \rho(x) $ for all $x\in B_{\delta'}$. For this it suffices to show that $\rho(F(x)) \leq \alpha \rho(x)$ for all $x\in B_{\delta'}$.

Define $F_1:B_\delta \rightarrow \mathbb{R}^r$ and $F_2:B_\delta \rightarrow \mathbb{R}^{p-r}$ by $F = (F_1,F_2)$. Define $T_1,T_2$ similarly.
Our choice of $\delta'$ guarantees that $\Vert DF_1 - T_1 \Vert_2 \leq \epsilon $ on $B_{\delta'}$. Now since $T$ is diagonal, $\Vert T_1 \Vert_2 = \Vert Q_1 T Q_1\Vert_2 \leq \alpha - \epsilon$, so in fact we have $\Vert DF_1 \Vert_2 \leq \alpha$ on $B_{\delta'}$. 

Next observe that since $E_2$ is invariant, we have that $F_1(E_2 \cap B_\delta) = 0$. Then for $x\in B_{\delta'}$,
\begin{eqnarray*}
\rho(F(x)) = \Vert F_1(x)\Vert_2 &=& \Vert F_1(x_1,x_2) - F_1(0,x_2) \Vert_2 \\
&=& \left\Vert \int_{0}^1 DF_1(tx_1,x_2)\cdot x_1 \,\ud t \right\Vert_2 \\
&\leq & \int_{0}^1 \Vert DF_1(tx_1,x_2)\Vert_2 \Vert x_1\Vert_2 \,\ud t \leq \alpha \rho(x),
\end{eqnarray*}
as desired. $\square$

\section{Proof of Proposition \ref{prop:saddlePoint}}\label{appendix:saddlePoint}

First we state a helpful lemma.

\begin{lemma}\label{lem:derivatives}
Let $Q=Q(t)$ be a differentiable density matrix-valued function of a
single variable. Then (omitting dependence on $t$ from the notation)
\[
(\wt{B}[Q])' = (B-\wt{B}[Q])Q'B^{-1} \wt{B}[Q] + \mathrm{h.c.}
\]
\end{lemma}
\begin{proof}
The proof is just a recapitulation of the argument in Lemma \ref{lem:fB}.
\end{proof}

Now we prove Proposition \ref{prop:saddlePoint}.

\proofof{Proposition \ref{prop:saddlePoint}}
Compute (omitting dependence on $t$ from the notation):
\begin{eqnarray*}
(F(Q))' & = & \mathrm{Tr} \big[ (\mf{F}(Q))' \left(A+\wt{B}[Q]\right) \mf{F}(Q) +\mathrm{h.c.} \big] + \mathrm{Tr} \big[ \mf{F}(Q) \left(\wt{B}[Q]\right)' \mf{F}(Q)   \big] \\ 
& = & 2\cdot\mathrm{Tr} \big[ (\mf{F}(Q))' \left(A+\wt{B}[Q]\right) \mf{F}(Q) \big] + 2\cdot \mathrm{Tr} \big[ \mf{F}(Q) (B-\wt{B}[Q])Q' \mf{F}(Q)   \big].
\end{eqnarray*}
We have used Lemma \ref{lem:derivatives}, together with the fact that $\wt{B}[Q]$ agrees with $B$ on $\mathrm{Im}(\mf{F}(Q))$. The first term in the last expression turns out to be zero. (This is essentially the content of the Hellmann-Feynman theorem.) We verify this presently.

For $i=1,\ldots,N$, let $u_i = u_i(t)$ be orthonormal such that $u_1,\ldots,u_n$ forms a basis for $\mathrm{Im}(Q(t))$. Then 
\begin{eqnarray*}
\mathrm{Tr} \big[ (\mf{F}(Q))' \left(A+\wt{B}[Q]\right) \mf{F}(Q)  \big] & = & \sum_{i=1}^N u_i^* (\mf{F}(Q))' \left(A+\wt{B}[Q]\right) \mf{F}(Q) u_i \\
& = & \sum_{i=1}^n u_i^* (\mf{F}(Q))' \left(A+\wt{B}[Q]\right) u_i.
\end{eqnarray*}
Now $\mathrm{Im}(\mf{F}(Q))$ is an invariant subspace for $\left(A+\wt{B}[Q]\right)$, so for $i=1,\ldots,n$, $\left(A+\wt{B}[Q]\right) u_i$ is an element of $\mathrm{Im}(\mf{F}(Q))$. Therefore
\[
\mathrm{Tr} \big[ (\mf{F}(Q))' \left(A+\wt{B}[Q]\right) \mf{F}(Q)  \big] = \sum_{i=1}^n u_i^*\, \mf{F}(Q) (\mf{F}(Q))' \mf{F}(Q) \left(A+\wt{B}[Q]\right) u_i.
\]
Note that
\[
(\mf{F}(Q))'=(\mf{F}(Q)\mf{F}(Q))'=\mf{F}(Q)(\mf{F}(Q))' + (\mf{F}(Q))'\mf{F}(Q).
\]
Multiply both sides by $\mf{F}(Q)$ and rearrange the terms, we have
\[
\mf{F}(Q) (\mf{F}(Q))' \mf{F}(Q) = 0.
\]
Therefore
\[
\mathrm{Tr} \big[ (\mf{F}(Q))' \left(A+\wt{B}[Q]\right) \mf{F}(Q)  \big] = 0
\]
as claimed and 
\begin{equation}\label{eqn:Ffirstderiv}
(F(Q))' = 2\cdot \mathrm{Tr} \big[ \mf{F}(Q)\, (B-\wt{B}[Q])\,Q'\, \mf{F}(Q)   \big].
\end{equation}
Define $\Delta P := Q'(0)$ and evaluate at $t=0$. Note that
$\mf{F}(Q(0))=P_f$, we obtain
\[
(F(Q))'(0) = 2\cdot \mathrm{Tr} \big[ P_f\, (B-\wt{B}[P_f])\,(\Delta P)\, P_f   \big].
\]
But $P_f (B-\wt{B}[P_f]) =0$, so  $(F(Q))'(0) = 0$, as desired.

Next take another derivative of \eqref{eqn:Ffirstderiv} and evaluate at $t=0$ to find
\begin{eqnarray*}
(F(Q))''(0) & =&  2\cdot \mathrm{Tr} \big[ D\mf{F}_{P_f}[\Delta P]
\, (B-\wt{B}[P_f])\,(\Delta P)\, P_f   \big] \\
& & \quad + \ \  2\cdot \mathrm{Tr} \big[ P_f\, (B-\wt{B}[P_f])\,(\Delta P)\, D\mf{F}_{P_f}[\Delta P] \big] \\
& & \quad - \ \ 2\cdot \mathrm{Tr} \big[ P_f\, (\wt{B}[Q])'(0)\,(\Delta P)\, P_f   \big] + 2\cdot \mathrm{Tr} \big[ P_f\, (B-\wt{B}[P_f])\,Q''(0)\, P_f   \big].
\end{eqnarray*}
Since $P_f (B-\wt{B}[P_f]) =0$, the second and final terms vanish. Substituting in for $(\wt{B}[Q])'(0)$ via Lemma \ref{lem:derivatives} (and again using the facts $P_f (B-\wt{B}[P_f]) =0$ and $B^{-1}\wt{B}[P_f]P_f = P_f$), we obtain
\begin{eqnarray*}
(F(Q))''(0) & =&  2\cdot \mathrm{Tr} \big[ D\mf{F}_{P_f}[\Delta P]
\, (B-\wt{B}[P_f])\,(\Delta P)\, P_f   \big] \\
& & \quad + \ \ 2\cdot \mathrm{Tr} \big[ P_f\, (\Delta P) \,(\wt{B}[P_f] - B) \,(\Delta P)\, P_f   \big].
\end{eqnarray*}

For the rest of the proof, $u_i$ will always indicate $u_i(0)$.
Now we substitute in for $D\mf{F}_{P_f}[\Delta P]$ via Lemma \ref{lem:composedJacobianGeneral}. Since $(B-\wt{B}[P_f])u_i =0$ for $i=1,\ldots,n$, only the ``$\mathrm{h.c.}$'' term survives, yielding
\begin{eqnarray}
&& (F(Q))''(0) \\
&& \quad =\ \   2 \sum_{i=1}^n  \mathrm{Tr} \left[ u_i u_i^*\,(\Delta P)\, \big(Z_i^{P_f}\big)^*\, (B-\wt{B}[P_f])\,(\Delta P)\, P_f   \right] \nonumber\\
& & \qquad\qquad + \ \ 2\cdot \mathrm{Tr} \big[ P_f\, (\Delta P) \,(\wt{B}[P_f] - B) \,(\Delta P)\, P_f   \big] \nonumber\\
&& \quad = \ \  2 \sum_{i=1}^n u_i^*\,(\Delta P)\, (\wt{B}[P_f]-B)\, \left[P_f^\perp \left(A+\wt{B}[P_f] - \mu_i \right) P_f^\perp \right]^\dagger \, (B-\wt{B}[P_f])\,(\Delta P)\, u_i \nonumber \\
& & \qquad\qquad + \ \ 2 \sum_{i=1}^n u_i^* \,(\Delta P)\,(\wt{B}[P_f] - B) \,(\Delta P)\, u_i. \nonumber
\end{eqnarray}
Let $X$ be an eigenvector of $D\mf{H}_0^{P_f}$ with corresponding
eigenvalue $\sigma$, viewed as an element of $\mathbb{C}^{(N-n)\times
n}$, so $X=(0,\ldots,0,X_j,0,\ldots,0)$, where $X_j$ is an eigenvector
of $J_j^{P_f}$. Fix the path $Q(t) = \Phi(tX)$, then 
\[
(\Delta P) u_i = \delta_{ij} \left(\begin{array}{c}
0 \\
X_j
\end{array}\right),
\]
and
\begin{eqnarray*}
(F(Q))''(0) &=& \ -2 X_j^*\, (\wt{B}[P_f]-B)_{22}\,J_j^{P_f}\,X_j \ +\  2 X_j^*\, (\wt{B}[P_f]-B)_{22}\,X_j \\
& = &\  2(1-\sigma) X_j^*\, (\wt{B}[P_f]-B)_{22}\,X_j.
\end{eqnarray*}
Since $(\wt{B}[P_f]-B)_{22} \succ 0$, we have $X_j^*\, (\wt{B}[P_f]-B)_{22}\,X_j > 0$, and therefore the sign of $(F(Q))''(0)$ is the sign of $1-\sigma$. 
The proposition is proved by recalling Lemma~\ref{lem:eigenJGeneral}.
$\square$

\section{Proof of the egg on barn lemma}\label{sec:eggOnBarn}

This section is devoted to the proof of Lemma \ref{lem:preimage}, which we break into several pieces.

First we outline some notation that will allow us to treat the Hermitian
and real-symmetric cases jointly. Fix $K\in \{\mathbb{C},\mathbb{R}\}$.
Let $\mathbf{K}_N$ denote $\mathbf{H}_N$ if $K=\mathbb{C}$ and
$\mathbf{S}_N$ if $K=\mathbb{R}$. Let $\mathbf{E}_N \subset
\mathbf{K}_N$ denote the elements of $\mathbf{K}_N$ with no repeated eigenvalues.

Fix some $B\in \mathbf{K}_N$ for the remainder of the section. We equip
$\mathbf{K}_N$ with the Lebesgue measure, so statements about, e.g.,
`almost every' $A$ in $\mathbf{K}_N$ should be understood with respect
to this measure. Meanwhile, we equip $\mathcal{D}_K$ with the natural
notion of `measure zero' inherited from the Lebesgue measure on charts,
which coincides with that of its volume measure induced by any choice
Riemannian metric.

Let $\Phi : \mathcal{D}_K \rightarrow \mathbf{K}_N$ denote the map $Q \mapsto A+\wt{B}[Q]$, and let $\Psi : \mathbf{E}_N \rightarrow \mathcal{D}_K$ denote the map that sends a matrix in $\mathbf{E}_N$ to its density matrix in $\mathcal{D}_K$. (Note that the choice of density matrix is unambiguous when there are no repeated eigenvalues.) We would like to say that $\Phi^{-1}(\mathbf{E}_N)$ is a large (i.e., full-measure) subset of $\mathcal{D}_K$, so that we can define $\Psi \circ \Phi$ (which coincides with $\mf{F}$) on this set. This is quite essential to the argument. Indeed, if this were not the case, then there would be a set of positive measure in $\mathcal{D}_K$ on which the behavior of $\mf{F}$ was not canonically determined, much less differentiable. Fortunately, we have the following lemma, which says even more.

\begin{lemma}\label{lem:largeOpen}
For almost every choice of $A$ in $\mathbf{K}_N$, $\mathcal{W}:= \Phi^{-1}(\mathbf{E}_N)$ is a connected open subset of full measure in $\mathcal{D}_K$.
\end{lemma}
\begin{proof}
The openness of $\mathcal{W}$ follows from the fact that $\Phi$ is a
continuous map $\mathcal{D}_K \rightarrow \mathbf{K}_N$ and that
$\mathbf{E}_N$ is open in $\mathbf{K}_N$.

Next note that a Hermitian (in particular, real-symmetric) matrix $X$
has repeated eigenvalues if and only if the discriminant of the
characteristic polynomial of $X$ is zero. This is a real-algebraic
condition on the entries of $X$ (with the real and complex parts treated
separately in the case $K=\mathbb{C}$), so $\mathbf{K}_N \backslash
\mathbf{E}_N$ is a real algebraic subset of the real vector space
$\mathbf{K}_N$. In fact (see Section 1.3 of \cite{TaoRMT}),
$\mathbf{K}_N \backslash \mathbf{E}_N$ has real codimension 3 in
$\mathbf{K}_N$ if $K=\mathbb{C}$ and real codimension 2 in
$\mathbf{K}_N$ if $K=\mathbb{R}$. Thus (since $\mathbf{K}_N \backslash
\mathbf{E}_N$ is a real algebraic set), in either case $\mathbf{K}_N
\backslash \mathbf{E}_N$ can be written as a (disjoint) union of
finitely many smooth submanifolds $M_1,\ldots,M_k$ of $\mathbf{K}_N$,
each of real codimension at least $2$ in $\mathbf{K}_N$.

Ideally, this should indicate that $\Phi^{-1}(\mathbf{K}_N \backslash \mathbf{E}_N)$ is a union of finitely many smooth submanifolds of $\mathcal{D}_K$, each of real codimension at least $2$. Indeed, we have by the Transversality Theorem (see, e.g., Section 2.3 of \cite{GuilleminPollack}) that for almost every $A\in \mathbf{K}_N$, the map $\Phi$ is transversal to $M_i$ for each $i=1,\ldots,k$. Then by the preimage theorem for transversal maps (see, e.g., Section 1.4 of \cite{GuilleminPollack}), $\Phi^{-1}(M_i)$ is a submanifold of $\mathcal{D}_K$ with (real) codimension in $\mathcal{D}_K$ equal to the codimension of $M_i$ in $\mathbf{K}_N$, which is at least $2$.

Thus $\mathcal{W}=\Phi^{-1}(\mathbf{E}_N)$ is equal to $\mathcal{D_K}$
minus a finite number of submanifolds of codimension at least $2$. These submanifolds have zero measure in $\mathcal{D}_K$ (this follows from Sard's theorem; refer, e.g., to \cite{GuilleminPollack}), so $\mathcal{W}$ has full measure in $\mathcal{D}_K$.

It only remains to show that $\mathcal{W}$ is connected. Since $\mathcal{D_K}$ is connected, this follows from the general fact that if $Y$ is a connected (hence smoothly path-connected) manifold and $Y_1,\ldots,Y_k$ are submanifolds with codimension at least 2 in $Y$, then $Y\backslash \bigcup_{i} Y_i$ is connected.

This general fact also follows from a transversality argument, which we
now provide for completeness. Let $x,y\in Y$, and let
$\gamma:[0,1]\rightarrow Y$ be a smooth path with $\gamma(0)=x$ and
$\gamma(1)=y$. But there is a homotopy of maps $\gamma_{\epsilon}$ (with
$\gamma_0 = \gamma$) such that $\gamma_{\epsilon}$ is transversal to each
of the $Y_i$ for a.e. $\epsilon$ (see, e.g., the proof of the
``transversality homotopy theorem'' of Section 2.3 of \cite{GuilleminPollack}). Since the $Y_i$ have codimension $2$, this implies that $\gamma_\epsilon$ does not intersect any of the $Y_i$ (for a.e. $\epsilon$). Taking $\epsilon$ sufficiently small so that $x=\gamma(0)$ and $y=\gamma(1)$ are connected to $\gamma_\epsilon(0)$ and $\gamma_\epsilon(1)$, respectively, by paths within $Y\backslash \bigcup_{i} Y_i$, we see that $x$ and $y$ are connected by a path within $Y\backslash \bigcup_{i} Y_i$.) Also, we know that $\Phi^{-1}(\mathbf{E}_N)$ is open in $\mathcal{D}_K$ because $\mathbf{E}_N$ is open in $\mathbf{K}_N$.
\end{proof}

We now outline the main pieces remaining in the proof of Lemma
\ref{lem:preimage}. Recall that $\mathcal{D}_K$ is a real-analytic
submanifold\footnote{$\mathcal{D}_K$ can be identified with the
Grassmannian $\mathbf{Gr}(n,K^N)$, i.e., the set of all $n$-dimensional
subspaces of $K^N$, which is an algebraic variety of $K$-dimension
$(N-n)n$. The space $\mathcal{D}_K$ itself is cut out by the conditions
$Q^2 = Q$, $Q^* = Q$, and $\mathrm{Tr}(Q) = n$ on $Q\in K^{N\times N}$.
These are real algebraic conditions on $K^{N\times N} \simeq
\mathbb{R}^{m}$ (for some $m$), so $\mathcal{D}_K$ is a (smooth) real
algebraic subvariety of $\mathbb{R}^{m}$. In particular, $\mathcal{D}_K$
has the structure of a real-analytic manifold.} of $\mathbb{R}^m$ for
some $m$. We claim that $\Phi$ is real-analytic on $\mathcal{D}_K$ and
that $\Psi$ is real-analytic on $\mathbf{E}_N$. This would imply that
$\mf{F}$ is a real-analytic map $\mathcal{W} \rightarrow \mathcal{D}_K$.
In particular, by an analytic continuation argument
(Lemma~\ref{lem:generalAnalyticFact}), the set
$\mathcal{W}'$ of points in $\mathcal{W}$ at which the Jacobian of
$\mf{F}$ fails to be invertible must either be all of $\mathcal{W}$ or
have zero measure in $\mathcal{D}_K$. The former possibility can be
ruled out.

Then in words, $\mf{F}$ is a local diffeomorphism on an open set of full measure in $\mathcal{D}_K$. Diffeomorphisms preserve measure zero sets, and by covering $\mathcal{W}'$ with countably many small open sets on which $\mf{F}$ is a diffeomorphism, we will see that the preimage of a measure zero set under $\mf{F}$ must have measure zero.

First we turn to establishing the claimed real-analyticity.

\begin{lemma}\label{lem:analytic}
$\mf{F}\vert_{\mathcal{W}}: \mathcal{W} \rightarrow \mathcal{D}_K$ is a real-analytic map between real-analytic manifolds.
\end{lemma}

\begin{proof}
For $Q \in \mathcal{D}_K$ (so in particular $Q=Q^*$), we can write 
\begin{eqnarray*}
\Phi(Q) &=& A+B(QBQ)^\dagger B \\ 
&=& A+ \frac{1}{2} B  \left[\big[QBQ+(I-Q)\big]^{-1} - (I-Q)\right] B \\
&& \quad\quad  +\ \frac{1}{2} B \left[\big[Q^*BQ^*+(I-Q^*)\big]^{-1} - (I-Q^*)\right] B.
\end{eqnarray*}
Written in the latter form, it is clear that $\Phi$ extends to a real-analytic map to $\mathbf{E}_N$, defined on a neighborhood of $\mathcal{D}_K$ in $K^{N\times N}$ (considered, in either case for $K$, as a real coordinate space $\mathbb{R}^q$ for some $q$).

Consider $X_0 \in \mathbf{E}_N$, and let $\mathcal{C}$ be a simple contour in the complex plane surrounding only the lowest $n$ eigenvalues of $X$. The for all $X$ in a sufficiently small neighborhood of $X_0$ in $\mathbf{E}_N$, we have
\[
\Psi(X) = \frac{1}{2\pi\I} \oint_{\mc{C}} (z-X)^{-1} \ud z,
\]
where $\mathcal{C}$ is a simple contour in the complex plane surrounding only the lowest $n$ eigenvalues of $X$. In particular, we can choose $\mathcal{C}$ to be a circle of some radius $R>0$, so taking the parametrization $z(t) = R\cos(t) + iR\sin(t)$ yields
\[
\Psi(X) = \int_0^{2\pi} \underbrace{\frac{R}{2\pi\I} (R\cos(t) + i\sin(t) -X)^{-1}}_{=:\,G(X,t)\,} \ud t.
\]
Identifying the real vector space $\mathbf{K}_N$ 
with $\mathbb{R}^p$ for some $p$, we have that $G(X,t)$ is a rational function $\mathbb{R}^{m+1} \rightarrow \mathbb{C}^{N\times N}$, well-defined for all $X$ in a neighborhood of $X_0$, hence real-analytic (if we identify the target space with $\mathbb{R}^{2N^2}$). Since an integral of a real-analytic function with respect to one of its arguments is real-analytic (see Proposition 2.2.3 of \cite{KrantzParks}), we have established that $\Psi$ is a real-analytic function $\mathbf{E}_N \rightarrow \mathcal{D}_K$, where we can interpret the domain as sitting inside some $\mathbb{R}^p$ and the target as sitting inside of $\mathbb{R}^m$, as mentioned above.

Since the composition of real-analytic functions is real-analytic (see Proposition 2.2.8 of \cite{KrantzParks}), we have established that $\mf{F}=\Psi \circ \Phi$ is real-analytic on a neighborhood of $\mathcal{W}\subset \mathcal{D}_K$ in $K^{N\times N}$. Since $\mathcal{W}$ is open in $\mathcal{D}_K$, $\mathcal{W}$ is a real-analytic submanifold of $\mathcal{D}_K$, and we can view $\mf{F}:\mathcal{W}\rightarrow \mathcal{D}_K$ as a real-analytic map between real-analytic manifolds (without thinking of their ambient spaces).
\end{proof}

Next we prove a general fact about real-analytic maps between
real-analytic manifolds. This is essentially an analytic continuation result.
\begin{lemma}\label{lem:generalAnalyticFact}
Suppose that $F:\mathcal{M}\rightarrow \mathcal{N}$ is a real-analytic
map between real-analytic manifolds of equal dimension $k$, and $\mathcal{M}$ is connected. Let $\mathcal{M}'$ be the closed subset of points $x\in \mathcal{M}$ at which the Jacobian $DF_x:T_x \mathcal{M} \rightarrow T_{F(x)} \mathcal{N}$ is singular. Then either $\mathcal{M} = \mathcal{M}'$ or $\mathcal{M}'$ has zero measure in $\mathcal{M}$.
\end{lemma}
\begin{proof}
In a local coordinate chart, the defining condition for $\mathcal{M}'$
is precisely that the determinant of the $k\times k$ Jacobian matrix in
local coordinates (whose entries are real-analytic functions of the
local coordinates) is zero. This set is a real-analytic function of local coordinates. The zero set of a real-analytic function on a connected open subset of $\mathbb{R}^k$ is either the whole set or a set of measure zero (in fact, by a much deeper result of Lojasiewicz, a finite union of analytic submanifolds of codimension at least $1$---see Theorem 6.3.3 of \cite{KrantzParks}).

Suppose that the measure of $\mathcal{M}'$ is not zero, so $\mathcal{M}'$ must have positive measure in some coordinate chart, and by the preceding $\mathcal{M}'$ must contain some open set in this chart.
Note that the set $\mathcal{A}:=\{x\in \mathcal{M} \,:\,DF_x\,\textrm{is
singular on a neighborhood of }x\}$ is both open and closed in
$\mathcal{M}$. The openness follows immediately from the definition, while the closedness follows from the real-analyticity of $F$. To see the latter point, let $y$ be a limit point of $\mathcal{A}$, and let $(\mathcal{U},\varphi)$ be a coordinate chart near $y$ with $\varphi(y) = 0$. Let the determinant of the $k\times k$ Jacobian matrix $DF$ in
local coordinates be denoted by $f$, so $f$ is real-analytic, and moreover $f\equiv 0$ on $\varphi(\mathcal{U})$. Then all of the derivatives of $f$ are uniformly zero on $\varphi(\mathcal{U})$, hence also at the limit point $0=\varphi(y)$. Since $f$ is real-analytic at $0$, this implies that $f \equiv 0$ on a neighborhood of $0$, so $DF_x$ is singular on a neighborhood of $y$, i.e., $y\in\mathcal{A}$. This establishes that $\mathcal{A}$ is closed, as desired.

Since $\mathcal{M}$ is connected and $\mathcal{A}$ is both open and closed in $\mathcal{M}$, we must have either $\mathcal{A}=\emptyset$ or $\mathcal{A}=\mathcal{M}$.
Since $\mathcal{M}'$ contains an open set, 
$\mathcal{A}$ cannot be empty. Consequently when the measure of
$\mathcal{M}'$ is not zero, $\mathcal{A}=\mathcal{M}$ and $DF_x$ is singular for all $x$.  \end{proof}

In particular, Lemma \ref{lem:generalAnalyticFact} implies (together with Lemma \ref{lem:largeOpen}) that 
\[
\mathcal{W}' := \{ Q\in \mathcal{W} \,:\,D\mf{F}_Q\ \textrm{is singular}\}
\]
is either equal to $\mathcal{W}$ or has zero measure in $\mathcal{W}$. The next lemma says that we can rule out the former possibility.

\begin{lemma}
For almost every choice of $A$ in $\mathbf{K}_N$, $\mathcal{W}'$ has zero measure in $\mathcal{W}$, hence also (by Lemma \ref{lem:largeOpen}) zero measure in $\mathcal{D}_K$. It follows that $\mathcal{W}\backslash \mathcal{W}'$ is an open subset of full measure in $\mathcal{D}_K$.
\end{lemma}
\begin{proof}
We only need to rule out the possibility that $\mathcal{W}' =
\mathcal{W}$. We will do so by considering a point near the true density matrix $P$. 

Recall from our proof of local convergence that $D\mf{F}_P$ has positive eigenvalues (with $\mf{F}$ considered, depending on the case for $K$, as either a map $\mathcal{D}_{\mathbb{C}}\rightarrow\mathcal{D}_{\mathbb{C}}$ or $\mathcal{D}_{\mathbb{R}}\rightarrow\mathcal{D}_{\mathbb{R}}$), hence is nonsingular. This means that $P \notin \mathcal{W}'$.

If $P \in \mathcal{W}$, then $\mathcal{W}' \neq \mathcal{W}$, and we are done. More generally, even if $P \notin \mathcal{W}$, observe that since $\mathcal{W}$ is of full measure in $\mathcal{D}_K$ (hence dense in $\mathcal{D}_K$), there is a sequence of density matrices $Q_j \rightarrow P$ with $Q_j \in \mathcal{W}$. Since $D\mf{F}_P$ is nonsingular, it follows that $D\mf{F}_{Q_j}$ is nonsingular for $j$ sufficiently large. But then $Q_j \in \mathcal{W}$ and $Q_j \notin \mathcal{W}'$, so $\mathcal{W}' \neq \mathcal{W}$, as desired.
\end{proof}

Now we finish the proof of Lemma \ref{lem:preimage} by the lemma below.
\begin{lemma}\label{lem:generalAnalyticFact2}
Let $F:\mathcal{M} \rightarrow \mathcal{N}$ be a map between smooth
manifolds of equal dimension, and let $\mathcal{V}$ be an open subset of
full measure in $\mathcal{M}$ on which $F$ is smooth and $DF$ is
nonsingular. Then for any set $S$ of measure zero in $\mathcal{N}$,
$F^{-1}(S)$ has measure zero in $\mathcal{M}$.
\end{lemma}
\begin{proof}
Now for every point in $x \in \mathcal{V}$, by the inverse function
theorem we can find a neighborhood $\mathcal{U}_x\ni x$ in $\mathcal{M}$
such that $F:\mathcal{U}_x \rightarrow \mathcal{N}$ is a diffeomorphism
onto its image. Moreover, the size of the neighborhood can be taken to
depend only on the derivatives of $F$ near $x$. In particular, we can
assume that the size of the neighborhood $\mathcal{U}_x$ is locally
bounded away from zero. (By this we mean, fixing some arbitrary Riemannian metric, that for every $x\in \mathcal{V}$, we can take
$\mathcal{U}_x$ to contain a Riemannian ball of radius $r(x)$ about $x$, where $x\mapsto r(x)$ is bounded
away from zero on every compact subset of $\mathcal{V}$).
By fixing a set of coordinate charts on the submanifold $\mathcal{V}$
and taking $\mathcal{X}$ to consist of all the $x$ that are rational
points in any of these coordinate charts, we see that
$\{\mathcal{U}_x\,:\,x\in\mathcal{X}\}$ forms a countable open cover of
$\mathcal{V}$. We remark that the details of this construction are made
quite explicit in order to avoid invoking the axiom of choice.

Let $S$ be a set with measure zero in $\mathcal{N}$. We can write
\begin{eqnarray*}
F^{-1}(S) &\subset& 
\left(\mathcal{M} \backslash \mathcal{V}\right) \cup
\left(F^{-1}(S)\cap \mathcal{V}\right) \\
& = & \left(\mathcal{M} \backslash \mathcal{V}\right) \cup \bigcup_{x\in\mathcal{X}} \left(F^{-1}(S)\cap \mathcal{U}_x \right) \\
& = & \left(\mathcal{M} \backslash \mathcal{V}\right) \cup \bigcup_{x\in\mathcal{X}} F^{-1}\left(S \cap F(\mathcal{U}_x) \right).
\end{eqnarray*}
Now the restriction of $F^{-1}$ to $F(\mathcal{U}_x)$ is a diffeomorphism, so $F^{-1}\left(S \cap F(\mathcal{U}_x)\right)$ is the diffeomorphic image of a measure zero set, hence has measure zero. As a countable union of measure zero sets, $F^{-1}(S)$ has measure zero.
\end{proof}

\bibliographystyle{siam}
\bibliography{aceconv}

\end{document}